\documentclass[12pt]{article}
\usepackage{geometry}
\usepackage{graphicx}
\usepackage{amsmath,amssymb,amsthm}
\usepackage{color}
\usepackage{soul}
\usepackage{lineno,hyperref}

\providecommand{\bysame}{\leavevmode ---\ }
\providecommand{\og}{``}
\providecommand{\fg}{''}
\providecommand{\smfandname}{et}

\providecommand{\smfedname}{\'ed.}

\newtheorem{theorem}{Theorem}[section]
\newtheorem{lemma}{Lemma}[section]
\newtheorem{corollary}{Corollary}[section]

\def\O{\Omega}
\def\o{\omega}
\def\e{\epsilon}
\def\l{\lambda}
\def\G{\Gamma}
\def\R{\mathbb{R}}
\def\a{\alpha}
\def\b{\beta}
\def\g{\gamma}
\def\p{\partial}

\def\T{\mathcal{T}}
\def\P{\mathcal{P}}
\def\E{\mathcal{E}}
\def\N{\mathcal{N}}
\def\Q{\mathcal{Q}}
\def\Cb{C_{\beta}}

\def\Bl{\left\lbrack\!\!\left\lbrack}
\def\Br{\right\rbrack\!\!\right\rbrack}

\begin{document}

\begin{center}

\textbf{\large Interpolation in Jacobi-weighted spaces and its application to a posteriori error estimations of the p-version of the finite element method}

\vspace{2cm}

\textbf{Mar\'{\i}a G.\ Armentano, Ver\'onica Moreno}
 
\vspace{1cm}

{Departamento de Matem\'atica,
Facultad de Ciencias Exactas y Naturales, Universidad de Buenos Aires,
IMAS - Conicet,  1428, Buenos Aires, Argentina.}
\end{center}

\begin{abstract}
The goal of this work is to introduce a local and a global interpolator in Jacobi-weighted spaces, with optimal order of approximation in the context of the $p$-version of finite element methods. Then, an a posteriori error indicator of the residual type is proposed for a model problem in two dimensions and, in the mathematical framework of the Jacobi-weighted spaces, the equivalence between the estimator and the error is obtained on appropriate weighted norm.

\end{abstract}

\section{Introduction}
\label{sec;introduction}

In this paper we show several results concerning the two dimensional Jacobi-weighted spaces and we introduce a local and a global interpolator with optimal order of approximation in the context of the $p$-version of finite element methods (FEM). Then, we consider a two dimensional model problem and we introduce an a posteriori error estimator of the residual type for the $p$-version of FEM, and we prove that the estimator is equivalent to the error on appropriate Jacobi-weighted norm up to higher order terms.

It is well known that the development of a posteriori error indicators and adaptive procedures play nowadays a relevant role in the numerical solution of partial differential equations. In contrast to the case of $h$ refinement, it seems to be an open question whether uniform reliability and efficiency can be achieved for an $hp$ a posteriori estimator of the residual type even in simple problems.

In the one dimensional case, analysis for a posteriori error indicators based on the residuals for the $p$ and $hp$ of FEM is well known \cite{Schwab1998, GuiBabuska1986III,DorflerHeuveline2007}. More precisely, in \cite{DorflerHeuveline2007}  the authors obtain an error estimator of the residual type for the Poisson Problem and prove that the $H^1$ norm of the error is equivalent to the error estimator up to higher order terms. Moreover, they propose an adaptive algorithm  and, since the error estimator is reliable and  efficient, they  prove that  the algorithm leads to a uniform monotone decrease of the energy error in every step. It is important to point out that these kind of results have not been established in high dimensions, and the techniques used for one-dimensional analysis can not be applied to higher dimensions.

In the two dimensional case, to the best of the authors' knowledge, the error estimators of the residual type present in the literature for the $p$ and $hp$ of FEM are equivalent to the error with constants depending on $p$ (see \cite{APRS2011,APRS2012,APS2014,MelenkWohlmuth2001} and the references therein). In particular, in  \cite{MelenkWohlmuth2001} the authors obtain  an error estimator of the residual type for the Poisson Problem in two dimensions and, using optimal weighted inverse estimations, prove that the error is equivalent to the indicator and propose an $hp$ strategy based on a predictor of the error in each element of the mesh. Using this error indicators, a particular algorithm is proposed in \cite{BurgDorfler2011} and  its convergence is reached assuming a data saturation  which, due to the constant $p$ dependence, becomes more restrictive for increasing polynomial degrees. In \cite{AinsworthSenior1997, AinsworthSenior1998} the authors proposed an $hp$ refinement strategy  in which in every step and for every element they decided to do $h$ adaptivity or $p$ adaptivity based in the local regularity of the solutions. On the other hand, in \cite{BraessPillweinSchoberl2009,DorsekMelenk2013} $p$-robust equilibrated residual error estimates are obtained for the Poisson problem and Elasticity problem respectively.

In recent decades, the Jacobi-weighted Sobolev spaces have received increasing attention for the approximation theory of the $p$ (and $hp$) version of the finite element methods. These spaces seem to be the appropriate functional spaces for a priori error analysis and play a crucial role in the analysis of the a posteriori error estimations (see \cite{AinsworthPinchedez2002,BabuskaGuo2000, BabuskaGuo2001, BabuskaGuo2002,BabuskaGuo2013}, and the references therein).
Indeed, the a priori error analysis and optimal convergence for the $p$-version of FEM in this context have been studied by several authors (see, for instance, \cite{BabuskaGuo2001, BabuskaGuo2002, GuoSun2007}). More recently, increasing attention to this framework has developed because of the need for optimal a posteriori error estimates \cite{BabuskaGuo2010,Guo2005}.
 Motivated by the results obtained by  \cite{DorflerHeuveline2007} in the one dimensional case and the a posteriori error analysis given by \cite{Guo2005}, in this paper we analyze the a priori and a posteriori approximation theory for the $p$-version in the mathematical framework of the Jacobi-weighted Sobolev spaces. In fact, we present several results concerning the interpolation theory for functions in Jacobi-weighted Sobolev spaces and, for the two dimensional Poison model problem, we develop an a posteriori error estimator of the residual type for the $p$-version of FEM. We analyze the equivalence of this estimator with the error in a Jacobi-weighted norm and we prove quasi-optimal global reliability and local efficiency estimates, both up to higher order terms. As far as we know, simultaneous reliability and efficiency estimates, both with constants independent of the polynomial degree,
have not been proved yet for any a posteriori error estimator for $p$ finite element methods in the two dimensional case and our estimates may be (see \cite{Guo2005,BabuskaGuo2013}) the best result that one can expected for error estimator based on residual in two dimensions.

The rest of the paper is organized as follows. In Section~\ref{Jacobi-spaces} we show several results concerning the Jacobi-weighted spaces. In Section~\ref{p-interpolation}
 we present the $p$-approximation theory and the interpolations error. In Section~\ref{A posteriori} we consider a two dimensional model problem and we introduce the a posteriori error estimator and we prove its equivalence with the error in a Jacobi-weighted norm.

\setcounter{equation}{0} \setcounter{table}{0}

\section{Jacobi-weighted Sobolev spaces}\label{Jacobi-spaces}

Let $Q=(-1,1)^2$ the reference domain in $\R^2$.
For $i=1,2$ let be $\b_i>-1$, $\a_i\geq 0$ integer,
$\b=(\b_1,\b_2)$ and $\a=(\a_1,\a_2)$. We define the weighted function $W_{\b, \a}$ in $Q$ as follows
\begin{equation*}
W_{\b, \a}(x,y)=(1-x^2)^{\b_1+ \a_1}(1-y^2)^{\b_2+ \a_2}.
\end{equation*}
If $\alpha=\mathbf{0}$ we note $W_{\b}=W_{\b,
\mathbf{0}}$.

For a function $u\in C^{\infty}(\bar{Q})$ and $k\geq 0$ integer, we define
the following norm:
\begin{equation*}
\|u\|^2_{H^{k,\b}(Q)}=\sum_{|\a|\leq k} \int_Q |\p^{\a} u|^2 W_{\b ,
\a}.
\end{equation*}
The weighted Sobolev space $H^{k,\b}(Q)$  is
defined as the closure of the  $C^{\infty}(\bar{Q})$ functions with this norm (see, for example, \cite{BabuskaGuo2001}), i.e.,
\begin{equation*}
H^{k,\b}(Q)=\overline{C^{\infty}(\bar{Q})}^{\| \cdot \|_{H^{k,\b}(Q)}} .
\end{equation*}
With $|u|_{H^{k,\b}(Q)}$  we denote the seminorms
\begin{equation*}
|u|_{H^{k,\b}(Q)}=\sum_{|\a|= k} \int_Q |\p^{\a} u|^2 W_{\b,
\a} .
\end{equation*}

Let $\O$ be an open polygonal domain in $\R^2$, $\T$ an admissible partition of
$\O$ in parallelograms. For any $K\in\T$, let $F:Q\rightarrow K$ be an affine transformation and  $u\in
C^{\infty}(\bar{K})$, then $\hat{u}=u\circ F \ \in
C^{\infty}(\bar{Q})$. We define
\begin{equation}\label{norma-en-K}
\|u\|_{H^{k,\b}(K)}=\|\hat{u}\|_{H^{k,\b}(Q)},
\end{equation}
and
\begin{equation*}
\|u\|^2_{H^{k,\b}(\T)}=\sum_{K\in\T}\|u\|^2_{H^{k,\b}(K)}.
\end{equation*}
We observe that the norm depends on the partition that we are considering. Then, the Jacobi-weighted spaces for $\T$, a
partition of $\O$, is defined as
\begin{equation*}
\begin{aligned}
H^{k,\b}(\T)&=\overline{C^{\infty}(\bar{\O})}^{\|\cdot\|_{H^{k,\beta}(\T)}}\\
H_0^{k,\b}(\T)&=\overline{C_0^{\infty}(\bar{\O})}^{\|\cdot\|_{H^{k,\beta}(\T)}}.
\end{aligned}
\end{equation*}

From now on, we consider the case $\b_i=\b >-1$ for $i=1,2$.

In what follows we present the Jacobi projection and we enunciate its properties  (see
\cite{GuoSun2007} and the references therein). In order to do that we need to introduce the Jacobi polynomials in one dimension (for details see \cite{Guo2009}).

Let $I=(-1,1)$, $\b>-1$ and let $p$ be a polynomial degree. For $x\in I$, let $J_p^{\b}(x)$ be the Jacobi polynomial of degree $p$, i.e.,
\begin{equation*}
J_p^{\b}(x)=\frac{(1-x^2)^{-\b}}{2^p
p!}\frac{d^{p}(1-x^2)^{\b+p}}{dx^p}.
\end{equation*}
It is well known that the Jacobi polynomials $J_p^{\b}(x)$ are orthogonal with the Jacobi weight $W_{\b}(x)=(1-x^2)^{\b}$, i. e.,
\begin{equation*}
\int_{I}J_p^{\b}(x)J_m^{\b}(x)W_{\b}(x)=\left\{\begin{aligned}\g_p^{\b},
\quad p&=m\\ 0, \quad p & \neq m\end{aligned}\right.
\end{equation*}
with
\begin{equation*}
\g_p^{\b}= \frac{2^{2\b+1}\G^2
(p+\b+1)}{(2p+2\b+1)\G(p+1)\G(p+2\b+1)},
\end{equation*}
where  $\G$ denotes the well known function Gamma given by
$ \G(z)=\int_0^{\infty} t^{z-1}e^{-t}dt$. 

Let $J_{p,k}^{\b}(x)=\frac{d^k}{dx^k} J_p^{\b}(x)$, then for
$0\leq k \leq p$ we have
\begin{equation*}
J_{p,k}^{\b}(x)=2^{-k}\frac{\G(p+2\b+k+1)}{\G(p+2\b+1)}J_{p-k}^{\b+k}(x)
\end{equation*}
which are orthogonal with the Jacobi weight $W_{\b+k}(x)$;
\begin{equation}\label{ortogonalidad-de-los-jacobi}
\int_{I}J_{p,k}^{\b}(x)J_{m,k}^{\b}(x)W_{\b+k}(x)=\left\{\begin{aligned}\g_{p,k}^{\b},
\quad &p=m\geq k\\ 0, \quad  &\mbox{otherwise}\end{aligned}\right.
\end{equation}
with
\begin{equation*}
\g_{p,k}^{\b}=\frac{2^{2\b+1}\G (p+2\b+k+1)\G^2
(p+\b+1)}{(2p+2\b+1)\G(p+1-k)\G^2(p+2\b+1)}.
\end{equation*}
We note that if $k=0$ we obtain $\g^{\b}_{p,0}=\g^{\b}_p$.

Now, we will enunciate two important properties that provide us an estimation for the constant $\g_{p,k}^{\b}$ and $\g_{p}^{\b}$. For this purpose we need some previous lemmas.

The following lemma (which can be found, for example, in page 427 of \cite{Lang1999}) gives a well known estimation for the function $\G$.

\begin{lemma}\label{Stirling-en-R}
For real $x$ and $x\rightarrow + \infty$ the following applies
\begin{equation*}
\G (x)  \sim x^{x-1/2}e^{-x}\sqrt{2\pi}
\end{equation*}
where $\sim$ means the quotient of the left  side by the right side  tents to 1 as $x\rightarrow + \infty$.
\end{lemma}

Thus, we have the following useful results.

\begin{corollary}\label{coro-stirling}
\begin{itemize}
\item[a)]For any $\a\in\R$, \quad   $\G (n+\a)  \sim (n+\a)^{n+\a-1/2}e^{-(n+\a)}\sqrt{2\pi}, \quad \mbox{ for } n\rightarrow + \infty$
\item[b)] For any $\a \in \R$, \quad $\lim_{n\rightarrow +\infty}\frac{\G (n+\a)}{\G(n)n^{\a}}=1,$
\item[c)] Given $\alpha_0 > -1 $, for any $\alpha$ such that $-1<\a\leq\a_0$, there exists positive constants  $A$ and $B$  (depending on $\alpha_0$ but independent of $\a$) such that
\begin{equation*}
A\leq \frac{\G (n+\a)}{\G(n)n^{\a}}\leq B \quad \forall n \in\mathbb{N}
\end{equation*}
\end{itemize}
\end{corollary}

Therefore, the following estimate holds.

\begin{lemma}\label{equivalencia-gamma}
Let $-1<\b\leq \b_0$, $k\geq 0$ integer and $p$ a polynomial degree. There exist positive constants  $A=A(\b_0)$ and $B=B(\b_0)$, independent of $p$ and $\b$,  such that
\begin{equation*}
A\g_{p}^{\b}p^{2k}\leq \g_{p,k}^{\b} \leq B p^{2k}\g_{p}^{\b} \quad\forall p \geq k.
\end{equation*}
\end{lemma}

\proof If $k=0$ we take $A=B=1$. Suppose $k\geq 1$.
\begin{equation*}
\begin{aligned}
\frac{\g_{p,k}^{\b}}{\g_{p}^{\b} p^{2k}}&=\frac{2^{2\b+1}\G (p+2\b+k+1)\G^2
(p+\b+1)}{(2p+2\b+1)\G(p+1-k)\G^2(p+2\b+1)} \frac{(2p+2\b+1)\G(p+1)\G(p+2\b+1)}{2^{2\b+1}\G^2
(p+\b+1)}\frac1{p^{2k}}\\
&=\frac{\G (p+2\b+k+1)\G(p+1)}{\G(p+1-k)\G(p+2\b+1)p^{2k}}\\
&= \frac{\G ((p+1-k)+(2\b+2k))}{\G(p+1-k)  (p+1-k)^{2\b+2k}}\frac{\G(p)p^{2\b+1}}{\G(p+2\b+1)}
\Big(\frac{p+1-k}{p}\Big)^{2\b+2k}\\
&=(I)(II)(III).
\end{aligned}
\end{equation*}
In $(I)$ and $(II)$ we apply Corollary  \ref{coro-stirling} c),  for $(III)$ we observe that
\begin{equation*}
\lim_{p\rightarrow \infty} \frac{p+1-k}{p} =1,
\end{equation*}
and, since $\frac{p+1-k}{p}>0$ for all $p \geq k$, we get
\begin{equation*}
c_1 \leq \frac{p+1-k}{p} \leq c_2, \quad \forall p\geq k,
\end{equation*}
for positive constants $c_1$ and $c_2$ and the proof concludes.
\endproof

\begin{lemma}\label{estimacion-gamma}
For $-1<\b<\b_0$ and $p\geq 1$ a polynomial degree there exist positive constants $A=A(\b_0)$ and $B=B(\b_0)$,  independent of  $p$ and $\b$, such that
\begin{equation*}
A p^{-1}\leq \g_{p}^{\b} \leq B p^{-1}.
\end{equation*}
\end{lemma}

\proof
\begin{equation*}
\begin{aligned}
\g_p^{\b}&= \frac{2^{2\b+1}\G(p+\b+1)^2}{(2p+2\b+1)\G(p+1)\G(p+2\b+1)}\\
&=\frac{2^{2\b+1}}{(2p+2\b+1)} \Big(\frac{\G(p+\b+1)}{\G(p)p^{\b+1}}\Big)^2\frac{p^{2\b+1}\G(p)}{\G(p+2\b+1)}
\end{aligned}
\end{equation*}
By Corollary \ref{coro-stirling} c) there exist positive constants  $c_1$ and $c_2$, independent of $p$ and $\b$, such that
\begin{equation*}
c_1\leq \Big(\frac{\G(p+\b+1)}{\G(p)p^{\b+1}}\Big)^2\frac{p^{2\b+1}\G(p)}{\G(p+2\b+1)}\leq c_2, \forall p\geq 1 , \forall -1<\b<\b_0
\end{equation*}
and the result holds.
\endproof

For any $\b>-1$, $k\geq0$ and $u\in H^{k,\b}(Q)$,  we have the Jacobi-Fourier expansion (see, for instance, \cite{GuoSun2007})
\begin{equation}\label{expancion-Jacobi}
u(x,y)=\sum_{i,j=0}^{\infty}c_{i,j}J_i^{\b}(x)J_j^{\b}(y)
\end{equation}
with
\begin{equation*}
c_{i,j}=\frac1{\g_i^{\b}\g_j^{\b}}\int_Q
u(x)J_i^{\b}(x)J_j^{\b}(y)W_{\b}(x,y).
\end{equation*}
Using the orthogonality  of the Jacobi polynomials
(\ref{ortogonalidad-de-los-jacobi}) we have that
\begin{equation}\label{integral-derivadas-por-peso}
\|u\|^2_{H^{0,\b}(Q)}=\sum_{i,j=0}^{\infty}
|c_{i,j}|^2\g_i^{\b}\g_j^{\b} \qquad
\mbox{and} \qquad
\int_Q |\p^{\a}u|^2 W_{\b,\a}=\sum_{i\geq
\a_1,j\geq\a_2}|c_{i,j}|^2\g_{i,\a_1}^{\b}\g_{j,\a_2}^{\b}.
\end{equation}

Then, for $-1<\b\leq \b_0$,  by Lemma \ref{equivalencia-gamma} we deduce that
\begin{equation*}
|u|^2_{H^{k,\b}(Q)}=\sum_{|\a|=k}\sum_{i\geq\a_1,j\geq\a_2}|c_{i,j}|^2\g_{i,\a_1}^{\b}\g_{j,\a_2}^{\b}
\simeq
\sum_{|\a|=k}\sum_{i\geq\a_1,j\geq\a_2}|c_{i,j}|^2\g_{i}^{\b}\g_{j}^{\b}i^{2\a_1}j^{2\a_2},
\end{equation*}
where $A\simeq B$ means $c_1 B\leq A \leq c_2 B$ with positive constants
$c_1$ and $c_2$  independent of $\b$.

For $p\geq 0$ we define $\Q_p(Q)$ the set of all polynomials of degree less or equal than $p$ in each variable in $Q$.  The Jacobi projection of $u$ in $\Q_p(Q)$ is
\begin{equation}\label{proyeccion-de-JacobiW}
\Pi_p^{\b} u (x,y)= \sum
_{i,j=0}^{p}c_{i,j} J_i^{\b}(x)J_j^{\b}(y).
\end{equation}

\begin{lemma}\label{cota-jacobi-en-vertices}
Let $-1<\b\leq \b_0$, then
\begin{equation*}
|J_p^{\b}(-1)|=|J_p^{\b}(1)|\sim  \frac1{\G(\b+1)} (p+1)^{\b}
\end{equation*}
with constants depending on $\b_0$ but independent of $p$ and $\b$.
\end{lemma}
\proof By equation (2.2) of \cite{GuoWang2004} we know that
\begin{equation}\label{jacobi-en-uno}
|J_p^{\b}(-1)|=|J_p^{\b}(1)|, \quad {\mbox{and}} \quad J_p^{\b}(1)=\frac{\G(p+\b+1)}{\G(p+1)\G(\b+1)}.
\end{equation}
Then, we can write
\begin{equation*}
\quad J_p^{\b}(1)=\frac{\G(p+\b+1)}{\G(p+1)(p+1)^{\b}}\frac{(p+1)^{\b}}{\G(\b+1)},
\end{equation*}
and the result follows from Corollary \ref{coro-stirling} c)
\endproof

The following theorem gives approximation properties of the Jacobi projection.
\begin{theorem}\label{Aproximacion-Jacobi-projection}
Let $I=(-1,1)$, $Q$ the reference domain in $\R^2$, $-1<\b\leq \b_0$, and $u\in
H^{1,\b}(Q)$. Let $p\geq 1$ a polynomial degree,  $\Pi^{\b}_p u \in \Q_p(Q)$
as in (\ref{proyeccion-de-JacobiW}). Then there exist a positive constant
$C=C(\b_0)$, independent  of $u$, $\b$ and $p$, such that
\begin{equation}\label{cota-interpolador-enQ}
\|u-\Pi^{\b}_p u\|_{H^{0,\b}(Q)} \leq C (p+1)^{-1}|u|_{H^{1,\b}(Q)}.
\end{equation}
If, in addition, $u\in C^0(\bar{Q})$  and $\b \leq -1/2$ then
\begin{equation}\label{cota-interpolador-en-lado}
\begin{aligned}
\|(u-\Pi^{\b}_p u)(\pm1,y)\|_{H^{0,\b}(I)} &\leq \frac{C}{\G(\b+1)} (p+1)^{-1/2}|u|_{H^{1,\b}(Q)},\\
\|(u-\Pi^{\b}_p u)(x,\pm1)\|_{H^{0,\b}(I)} &\leq \frac{C}{\G(\b+1)} (p+1)^{-1/2}|u|_{H^{1,\b}(Q)},
\end{aligned}
\end{equation}
and if, in addition, $\b < -1/2$ then
\begin{equation}\label{cota-interpolador-en-vertice}
|(u-\Pi^{\b}_p u)(V)| \leq \frac{C}{(-1-2\b)\G(\b+1)^2}(p+1)^{\b+1/2}|u|_{H^{1,\b}(Q)}, \ \forall \  V \mbox{ vertex of } Q.
\end{equation}

\end{theorem}
\begin{proof}
The proof of (\ref{cota-interpolador-enQ}) is given in \cite{GuoSun2007} with a constant that could depend on $\b$. However, following the steps of that proof, we  observe that a positive constant $C$ can be chosen independent of $\b$. Now, we prove  (\ref{cota-interpolador-en-lado}), the bound in the edges of $Q$ for smooth functions. We carry out the case  $\|(u-\Pi^{\b}_p u)(x,-1)\|_{H^{0,\b}(I)}$,  the other cases can be obtained analogously. Since $u\in C^0(\bar{Q})$ we can write
\begin{equation*}
\begin{aligned}
(u-\Pi^{\b}_p u)(x,-1)&=\big( \sum_{i\geq p+1,j\geq
p+1}+\sum_{i\geq p+1,j< p+1}+\sum_{ i< p+1,j\geq
p+1} \big) c_{i,j}J_i^{\b}(x)J_j^{\b}(-1)\\
& = \sum_{i\geq p+1}b_i^{[1]}J_i^{\b}(x)+\sum_{i\geq p+1}b_i^{[2]}J_i^{\b}(x)+\sum_{i< p+1}b_i^{[3]}J_i^{\b}(x)
\end{aligned}
\end{equation*}
where
\begin{equation*}
\begin{aligned}
b_i^{[1]}&=\sum_{j\geq p+1} c_{i,j}J_j^{\b}(-1),\quad i\geq p+1\\
b_i^{[2]}&=\sum_{j< p+1} c_{i,j}J_j^{\b}(-1),\quad i\geq p+1\\
b_i^{[3]}&=\sum_{j\geq p+1} c_{i,j}J_j^{\b}(-1),\quad i< p+1.
\end{aligned}
\end{equation*}
It is well known that, if $f:[p,+\infty]\rightarrow \R$ is a non negative, decreasing and integrable function then,
\begin{equation}\label{serie-por-integral}
\sum_{i=p+1}^{\infty}f(n) \leq \int_{p}^{\infty} f(x) dx.
\end{equation}
Therefore, by Lemma \ref{cota-jacobi-en-vertices}, H\"{o}lder inequality, Lemma  \ref{estimacion-gamma} and (\ref{serie-por-integral}) we get,
\begin{equation*}
\begin{aligned}
|b_i^{[1]}|^2&\leq \big(\sum_{j\geq p+1}|c_{i,j}||J_j^{\b}(-1)| \big)^2\\
&\leq \frac{C}{\G(\b+1)^2} \big(\sum_{j\geq p+1}|c_{i,j}|(j+1)^{\b}(\g_j^{\b})^{1/2}(\g_j^{\b})^{-1/2} j j^{-1}\big)^2\\
&\leq \frac{C}{\G(\b+1)^2} \big(\sum_{j\geq p+1}|c_{i,j}|^2\g_j^{\b}j^2\big)\big(\sum_{j\geq p+1}(j+1)^{2\b}(\g_j^{\b})^{-1}j^{-2} \big)\\
&\leq \frac{C}{\G(\b+1)^2}\big(\sum_{j\geq p+1}|c_{i,j}|^2\g_j^{\b}j^2\big)\big(\sum_{j\geq p+1}j^{-1+2\b} \big)\\
&\leq \frac{C}{\G(\b+1)^2}\big(\sum_{j\geq p+1}|c_{i,j}|^2\g_j^{\b}j^2\big)\big(\int_{p}^{\infty}x^{-1+2\b}dx \big)\\
&\leq \frac{C}{\G(\b+1)^2}p^{2\b} \sum_{j\geq p+1}|c_{i,j}|^2\g_j^{\b}j^2.
\end{aligned}
\end{equation*}
Then,
\begin{equation*}
\begin{aligned}
\|\sum_{i\geq p+1}b_i^{[1]}J_i^{\b}(x)\|^2_{H^{0,\b}(I)}&= \sum_{i\geq p+1}|b_i^{[1]}|^2\g_i^{\b}
\leq \frac{C}{\G(\b+1)^2} p^{2\b}\sum_{i\geq p+1,j\geq p+1}|c_{i,j}|^2\g_j^{\b}j^2 \g_i^{\b} \\
&\leq \frac{C}{\G(\b+1)^2} p^{2\b}\sum_{i\geq 0,j\geq 1}|c_{i,j}|^2\g_j^{\b}j^2 \g_i^{\b}
\leq \frac{C}{\G(\b+1)^2} p^{2\b} |u|^2_{H^{1,\b}(Q)}. 
\end{aligned}
\end{equation*}
Analogously,
\begin{equation*}
\begin{aligned}
|b_i^{[2]}|^2&\leq \big(\sum_{j< p+1}|c_{i,j}||J_j^{\b}(-1)| \big)^2\\
&\leq \frac{C}{\G(\b+1)^2} \big(\sum_{j< p+1}|c_{i,j}|(j+1)^{\b}(\g_j^{\b})^{1/2}(\g_j^{\b})^{-1/2} i i^{-1}\big)^2\\
&\leq\frac{C}{\G(\b+1)^2} (p+1)^{-2}\big(\sum_{j< p+1}|c_{i,j}|(j+1)^{\b}(\g_j^{\b})^{1/2}(\g_j^{\b})^{-1/2} i \big)^2\\
&\leq \frac{C}{\G(\b+1)^2} (p+1)^{-2} \big(\sum_{j< p+1}|c_{i,j}|^2\g_j^{\b}i^2\big)\big(\sum_{j< p+1}(j+1)^{2\b}(\g_j^{\b})^{-1} \big)\\
&\leq \frac{C}{\G(\b+1)^2}(p+1)^{-2} \big(\sum_{j< p+1}|c_{i,j}|^2\g_j^{\b}i^2\big)\big(\sum_{j< p+1}(j+1)^{2\b+1} \big)\\
&\leq \frac{C}{\G(\b+1)^2} (p+1)^{-2} \big(\sum_{j< p+1}|c_{i,j}|^2\g_j^{\b}i^2\big)(p+1).
\end{aligned}
\end{equation*}
Thus,
\begin{equation*}
\begin{aligned}
\|\sum_{i\geq p+1}b_i^{[2]}J_i^{\b}(x)\|^2_{H^{0,\b}(I)}&= \sum_{i\geq p+1}|b_i^{[2]}|^2\g_i^{\b}\\
&\leq  \frac{C}{\G(\b+1)^2} (p+1)^{-1} \sum_{i\geq p+1,j< p+1}|c_{i,j}|^2\g_j^{\b}i^2\g_i^{\b} \\
&\leq \frac{C}{\G(\b+1)^2} (p+1)^{-1} \sum_{i\geq 1,j\geq 0}|c_{i,j}|^2\g_j^{\b}i^2\g_i^{\b} \\
&\leq \frac{C}{\G(\b+1)^2}(p+1)^{-1} |u|^2_{H^{1,\b}(Q)}.
\end{aligned}
\end{equation*}
Similarly,
\begin{equation*}
\|\sum_{i< p+1}b_i^{[3]}J_i^{\b_1}(x)\|^2_{H^{0,\b}(I)}\leq \frac{C}{\G(\b+1)^2} p^{-1} |u|^2_{H^{1,\b}(Q)}.
\end{equation*}
As a consequence,
\begin{equation*}
\|(u-\Pi^{\b}_p u)\|_{H^{0,\b}(I)}\leq \frac{C}{\G(\b+1)}(p+1)^{-1/2} |u|_{H^{1,\b}(Q)}
\end{equation*}
and (\ref{cota-interpolador-en-lado}) holds.

Finally, we prove (\ref{cota-interpolador-en-vertice})  for $V=(-1,-1)$, the same arguments can be used for the other vertices of $Q$.
\begin{equation*}
\begin{aligned}
|(u-\Pi^{\b}_p u)(-1,-1)|&=|\big( \sum_{i\geq p+1,j\geq
p+1}+\sum_{i\geq p+1,j< p+1}+\sum_{ i< p+1,j\geq
p+1} \big) c_{i,j}J_i^{\b}(-1)J_j^{\b}(-1)|\\
&\leq \big( \sum_{i\geq p+1,j\geq
p+1}+\sum_{i\geq p+1,j< p+1}+\sum_{ i< p+1,j\geq
p+1} \big) |c_{i,j}||J_i^{\b}(-1)||J_j^{\b}(-1)|\\
&\leq\frac{C}{\G(\b+1)^2}\big( \sum_{i\geq p+1,j\geq
p+1}+\sum_{i\geq p+1,j< p+1}+\sum_{ i< p+1,j\geq
p+1} \big) |c_{i,j}|(i+1)^{\b}(j+1)^{\b}\\
&= \frac{C}{\G(\b+1)^2}(I+II+III).
\end{aligned}
\end{equation*}
At first we compute $II$ and $III$. Due to H\"{o}lder inequality and  Lemma  \ref{estimacion-gamma} we obtain
\begin{equation*}
\begin{aligned}
II&\leq \big(\sum_{i\geq p+1,j< p+1}|c_{i,j}|^2\g_i^{\b}\g_j^{\b}i^2\big)^{1/2}\big(\sum_{i\geq p+1,j< p+1}(i+1)^{2\b}(j+1)^{2\b}(\g_i^{\b})^{-1}(\g_j^{\b})^{-1}i^{-2}\big)^{1/2}\\
&\leq C \big(\sum_{i\geq p+1,j< p+1}|c_{i,j}|^2\g_i^{\b}\g_j^{\b}i^2\big)^{1/2}\big(\sum_{i\geq p+1,j< p+1}(i+1)^{2\b+1}(j+1)^{2\b+1}i^{-2}\big)^{1/2}\\
&\leq  C\big(\sum_{i\geq p+1,j< p+1}|c_{i,j}|^2\g_i^{\b}\g_j^{\b}i^2\big)^{1/2}\big(\sum_{i\geq p+1,j< p+1}(i+1)^{2\b+1}i^{-2}\big)^{1/2}\\
&\leq C \big(\sum_{i\geq p+1,j< p+1}|c_{i,j}|^2\g_i^{\b}\g_j^{\b}i^2\big)^{1/2}\big((p+1)\sum_{i\geq p+1}(i)^{2\b-1}\big)^{1/2}
\end{aligned}
\end{equation*}

Now, from (\ref{serie-por-integral}) we follow that
\begin{equation*}
\begin{aligned}
II
&\leq C \big(\sum_{i\geq p+1,j< p+1}|c_{i,j}|^2\g_i^{\b}\g_j^{\b}i^2\big)^{1/2}\big((p+1)\int_p^{\infty}x^{-1+2\b}dx\big)^{1/2}\\
&\leq C\big(\sum_{i\geq p+1,j< p+1}|c_{i,j}|^2\g_i^{\b}\g_j^{\b}i^2\big)^{1/2}\big((p+1)p^{2\b}\big)^{1/2}\\
&\leq C \big(\sum_{i\geq p+1,j< p+1}|c_{i,j}|^2\g_i^{\b}\g_j^{\b}i^2\big)^{1/2}\big(p^{2\b+1}\big)^{1/2} \leq C p^{\b+1/2}|u|_{H^{1,\b}(Q)}.
\end{aligned}
\end{equation*}

For $III$ we proceed analogously but changing the roles of $i$ and $j$, and we can conclude that
\begin{equation*}
III\leq C p^{\b+1/2}|u|_{H^{1,\b}(Q)}.
\end{equation*}
In what follows we compute $I$.
\begin{equation*}
\begin{aligned}
I&\leq  \big(\sum_{i\geq p+1,j\geq
p+1}|c_{i,j}|^2\g_i^{\b}\g_j^{\b}ij\big)^{1/2}\big(\sum_{i\geq p+1,j\geq
p+1}(i+1)^{2\b}(j+1)^{2\b}(\g_i^{\b})^{-1}(\g_j^{\b})^{-1}i^{-1}j^{-1}\big)^{1/2}\\
&\leq C \big(\sum_{i\geq p+1,j\geq
p+1}|c_{i,j}|^2\g_i^{\b}\g_j^{\b}ij\big)^{1/2}\big(\sum_{i\geq p+1}i^{2\b}\sum_{j\geq p+1}(j+1)^{2\b}\big)^{1/2}\\
&\leq C \big(\sum_{i\geq p+1,j\geq
p+1}|c_{i,j}|^2\g_i^{\b}\g_j^{\b}ij\big)^{1/2}\big(\int_p^{\infty}x^{2\b}dx \int_p^{\infty}y^{2\b}dy\big)^{1/2}.
\end{aligned}
\end{equation*}
Since $\b < -1/2$ the integrals converge, and it follows that
\begin{equation*}
\begin{aligned}
I&\leq \frac{C}{-1-2\b}\big(\sum_{i\geq p+1,j\geq
p+1}|c_{i,j}|^2\g_i^{\b}\g_j^{\b}ij\big)^{1/2}\big(p^{2\b+1}p^{2\b+1}\big)^{1/2}\\
&\leq \frac{C}{-1-2\b}p^{2\b+1} \big(\sum_{i\geq p+1,j\geq
p+1}|c_{i,j}|^2\g_i^{\b}\g_j^{\b}i^2+\sum_{i\geq p+1,j\geq
p+1}|c_{i,j}|^2\g_i^{\b}\g_j^{\b}j^2\big)^{1/2}\\
&\leq \frac{C}{-1-2\b}p^{2\b+1} |u|_{H^{1,\b}(Q)}.
\end{aligned}
\end{equation*}
Therefore,
\begin{equation*}
\begin{aligned}
|(u-\Pi^{\b}_p u)(-1,-1)|\leq \frac{C}{(-1-2\b)\G(\b+1)^2}p^{\b+1/2} |u|_{H^{1,\b}(Q)},
\end{aligned}
\end{equation*}
and the proof concludes.
\end{proof}

Let $K$ be a parallelogram of $\R^2$ and let $F_K : Q \rightarrow K$ be an affine transformation. For $p\geq
0$ we define

\begin{equation*}
\Q_p(K)=\{u|u\circ F_K \in \Q_p(Q) \}.
\end{equation*}

 Now, from (\ref{norma-en-K})  and   Theorem \ref{Aproximacion-Jacobi-projection} we have the following result.

\begin{corollary}\label{cotas-enK}
Let $K$ be a parallelogram of $\R^2$,  $-1<\b\leq \b_0$, and $u\in
H^{1,\b}(K)$. Let $p$ be a polynomial degree,  there exist $\Pi^{\b}_{p,K} u \in \Q_p(K)$ and a positive constant
$C$ independent of $p$, $\b$, and $u$ such that
\begin{equation}\label{cota-interpolador-enK}
\|u-\Pi^{\b}_{p,K} u\|_{H^{0,\b}(K)} \leq C (p+1)^{-1}|u|_{H^{1,\b}(K)},
\end{equation}
if, in addition, $u\in C^0(\bar{K})$  and $\b \leq -1/2$ then
\begin{equation}\label{cota-interpolador-en-lado-enK}
\|u-\Pi^{\b}_{p,K} u\|_{H^{0,\b}(\g)} \leq \frac{C}{\G(\b+1)} (p+1)^{-1/2}|u|_{H^{1,\b}(K)}\quad \forall \ \g \mbox{ edge of } K,
\end{equation}
if, in addition, $\b < -1/2$ then
\begin{equation}\label{cota-interpolador-en-vertice-enK}
|(u-\Pi^{\b}_p u)(V)| \leq \frac{C}{(-1-2\b)\G(\b+1)^2} (p+1)^{1/2+\b}|u|_{H^{1,\b}(K)}, \ \forall \ V \mbox{ vertex of } K.
\end{equation}
\end{corollary}

\section{$p$-Interpolation of smooth functions}\label{p-interpolation}

The goal of this section is to introduce interpolation operators which are suitable to obtain a residual error estimation in the mathematical framework of Jacobi-weighted Sobolev spaces of the $p$-version of Finite Element Methods.

Let $\O$ be an open polygonal domain in $\R^2$, $\T$ an admissible partition of
$\O$ in parallelograms. Let $p$ be a polynomial degree, we denote
 \begin{equation*}
\begin{aligned}
S^{p}(\T)&=\{u\in C^0(\O)\ |\ u|_K\in \Q_p(K) \quad \forall K \in \T\},\\
S_0^{p}(\T)&=\{u\in C_0^0(\O)\ |\ u|_K\in \Q_p(K) \quad \forall K \in \T\}.
\end{aligned}
\end{equation*}

 We choose a polynomial degree $p_K$ for each $K\in \T$ and we note $\mathbf{p}=(p_K)$ the vector of polynomials degrees. We assume that the polynomials degrees of neighboring elements are comparable, i.e., there exists a positive constant $C$ such that
\begin{equation*}
p_K\leq C p_{K'} \quad K,K' \in \T\mbox{ with } K\cap K'\neq \emptyset.
\end{equation*}
We introduce the following notation
\begin{equation}\label{espacio-de-elementos-finitos}
\begin{aligned}
S^{\mathbf{p}}(\T)&=\{u\in C^0(\O)\ |\ u|_K\in \Q_{p_K}(K)\},\\
S_0^{\mathbf{p}}(\T)&=\{u\in C_0^0(\O)\ |\ u|_K\in \Q_{p_K}(K)\},
\end{aligned}
\end{equation}
and
\begin{equation*}
\begin{aligned}
\E&=\{\mbox{all edges } e \mbox{ in } \T \},\\
\E^{\circ}&=\E\cap \O^{\circ},\\
\N&=\{\mbox{all vertices } V \mbox{ in } \T \}.
\end{aligned}
\end{equation*}
For $V\in \N$ we denote
\begin{equation*}
\begin{aligned}
\o_V &=\bigcup \{K|K\in \T \mbox{ and } K\cap V \neq\emptyset\},\\
\T_V&=\T|_{\o_V},\\
p_V &= \mbox{min}\{p_K|V\in K\},\\
\E_V&=\{\mbox{all edges } e \mbox{ of } \E\mbox{ such that } V \mbox{ is an endpoint of } e\}.
\end{aligned}
\end{equation*}
For any  $K\in\T$  or  $e\in \E$ we define 

\begin{equation*}
\begin{aligned}
\o_K &=\bigcup \{K'|K'\in T \mbox{ and } K'\cap K \neq\emptyset\}, \\
\o_e &=\bigcup \{K|K\in T \mbox{ and } K\cap e \neq\emptyset\},\\
p_e&= \mbox{min}\{p_K|e \mbox{ edge of } K\}.
\end{aligned}
\end{equation*}

Let $K\in\T$ and let $F_K :Q\rightarrow K$ be an affine transformation, for a function $u$ in $K$ we denote $\hat{u}=u\circ F_K$.
Let $\hat{e}=I\times\{-1\})$, then for any $e\in\E$ let $F_{e}:\hat{e}\rightarrow e$ be an affine transformation then, for a function $u$ in $e$ we denote $\hat{u}=u\circ F_{e}$ 

In order to introduce the local and the global interpolation operator we need some previous lemmas.

\begin{lemma}\label{existencia-del-poli-g}
Let $I=(-1,1)$, $-1<\b\leq \b_0$ and $p$ a polynomial degree, there exists $g\in\P_p(\bar{I})$ such that $g(1)=0$, $g(-1)=1$ and $$\|g\|_{H^{0,\b}(I)}\leq C \G(\b+1) (p+1)^{-(1+\b)},$$
where the constant $C=C(\b_0)$ is independent of $p$ and $\b$.
\end{lemma}

\begin{proof}
Let $J_{p-1}^{(\b+2,\b)}$ be the Jacobi polynomial of degree $p-1$ which is orthogonal with the weight $(1-x)^{\b+2}(1+x)^{\b}$ (see \cite{Guo2009} for details), we define
\begin{equation*}
g(x)=\frac{(1-x)J_{p-1}^{(\b+2,\b)}(x)}{2J_{p-1}^{(\b+2,\b)}(-1)},
\end{equation*}
then  $g\in \P_p (\bar{I})$, $g(-1)=1$, $g(1)=0$ and
\begin{equation*}
\begin{aligned}
\|g\|_{H^{0,\b}(I)}^2&=\frac1{4\big(J_{p-1}^{(\b+2,\b)}(-1)\big)^2}\int_{I} (1-x)^2\big(J_{p-1}^{(\b+2,\b)}(x)\big)^2(1-x)^{\b}(1+x)^{\b}\\
&=\frac1{4\big(J_{p-1}^{(\b+2,\b)}(-1)\big)^2}\int_{I} \big(J_{p-1}^{(\b+2,\b)}(x)\big)^2(1-x)^{2+\b}(1+x)^{\b}\\
&=\frac1{4\big(J_{p-1}^{(\b+2,\b)}(-1)\big)^2}\g_{p-1}^{\b+2,\b}.
\end{aligned}
\end{equation*}
As it has been shown in \cite{Guo2009} equation (2.2)
\begin{equation*}
J_{p-1}^{(\b+2,\b)}(-1)=\frac{(-1)^{p-1}\G(p+\b)}{(p-1)! \G(\b+1)},
\end{equation*}
and from Lemma \ref{cota-jacobi-en-vertices} we have that
\begin{equation*}
J_{p-1}^{(\b+2,\b)}(-1)\simeq  \frac{p^{\b}}{\G(\b+1)}.
\end{equation*}
Then,
\begin{equation*}
\|g\|_{H^{0,\b}(I)}^2\leq C \G(\b+1)^2 p^{-2\b}\g_{p-1}^{\b+2,\b},
\end{equation*}
and from equation (2.7) of \cite{Guo2009} we know that
\begin{equation*}
\begin{aligned}
\g_{p-1}^{\b+2,\b}&=\frac{2^{2\b+3}}{2p+2\b+1}\frac{\G(p+\b+1)\G(p+\b)}{\G(p)\G(p+2\b+2)}\\
&= \frac{2^{2\b+3}}{2p+2\b+1}\frac{\G(p+\b+1)}{\G(p)p^{\b+1}}\frac{\G(p+\b)}{\G(p)p^{\b}}\frac{\G(p)p^{2\b+2}}{\G(p+2\b+2)}\frac1{p}.
\end{aligned}
\end{equation*}
Now, by  Corollary \ref{coro-stirling} c)
\begin{equation*}
\frac{\G(p+\b+1)}{\G(p)p^{\b+1}} \leq C, \quad
 \frac{\G(p+\b)}{\G(p)p^{\b}} \leq C, \quad  \mbox{ and } \quad
 \frac{\G(p)p^{2\b+2}}{\G(p+2\b+2)} \leq C,
\end{equation*}
and hence,
\begin{equation*}
\g_{p-1}^{\b+2,\b} \leq C p^{-2},
\end{equation*}
and we conclude that
\begin{equation*}
\|g\|_{H^{0,\b}(I)}^2\leq C \G(\b+1)^2 p^{-2\b}p^{-2} =C \G(\b+1)^2p^{-2(\b+1)}
\end{equation*}
as claimed.
\end{proof}

\begin{corollary}\label{existencia-de-chi-l}
Let $K$ be a parallelogram and $V_1,V_2,V_3$ and $V_4$ its vertices. Let $-1<\b\leq\b_0$ and $p$ be a polynomial degree. There exist a function  $\xi_{K,l}$  and a positive constant $C=C(\b_0)$ independent of $p$ and $\b$ such that
\begin{itemize}
\item[i)] $\xi_{K,l}(V_j)=\delta_{lj}$ (takes the value $1$ in $V_l$ and $0$
in the others vertices),
\item[ii)] $\xi_{K,l}\in\Q_{p}(K)$ ,
\item[iii)] $\|\xi_{K,l}\|_{H^{0,\b}(K)}\leq C \G(\b+1)^2 (p+1)^{-2-2\b}$ and
\item[iv)] $\|\xi_{K,l}\|_{H^{0,\b}(e)}\leq C \G(\b+1) (p+1)^{-1-\b} \quad \forall  \ e \mbox{ edge of } K$.
\end{itemize}
\end{corollary}

\begin{proof}
Let $Q=(-1,1)^2$ the reference rectangle, $p$ a polynomial degree and $g$ as in the Lemma \ref{existencia-del-poli-g} then, the function $G(x,y)=g(x)g(y)$ is in $\Q_{p}(Q)$ and satisfies
\begin{align*}
G(-1,-1)&=1  \mbox{ and takes the value $0$ in the others vertices of $Q$}, \\
\|G\|_{H^{0,\b}(Q)}&\leq C \G(\b+1)^2 (p+1)^{-2(1+\b)}, \\
\|G\|_{H^{0,\b}(e)} &\leq C \G(\b+1) (p+1)^{-(1+\b)} \quad \forall \ e \mbox{ edge of } Q.
\end{align*}

Let $F_{K,V_l}:Q\rightarrow K$ be an affine transformation such that $F_{K,V_l}(-1,-1)=V_l$ then, $\xi_{K,l} = G \circ F_{K,V_l}^{-1}$ satisfies i)-iv) and the proof concludes.
\end{proof}

\begin{corollary}\label{extension}
Let $K$ be a parallelogram and $e$ an edge of $K$. Let $-1<\b\leq\b_0$ and $p$ be a polynomial degree and $w\in\P_p(e)$ such that $w|_{\partial e} = 0$ (i.e. $w(V)=0$ for all $V$ vertex of $e$). There exist $\psi$ an extension of $w$ to $K$
and a positive constant $C=C(\b_0)$, independent of $p$ and $\b$, such that
\begin{itemize}
\item[i)] $\psi \in \Q_p(K)$
\item[ii)] $\psi | _e = w $,
\item[iii)] $\psi | _{\partial K \setminus e} = 0$ and
\item[iv)] $\|\psi\|_{H^{0,\b}(K)} \leq C \G(\b+1) (p+1)^{-(1+\b)}\|w\|_{H^{0,\b}(e)}.$
\end{itemize}
\end{corollary}

\begin{proof}
Given $Q=(-1,1)^2$, the reference rectangle and $\hat{e}=(-1,1)\times \{-1\}$ an edge of $Q$, we consider $F_{K,e}:Q\rightarrow K$ an affine transformation such that $F_{K,e}(\hat{e})=e$. Hence, we define
 $\hat{w}\in\P_p(\hat{e})$ as $\hat{w}=w \circ F_{K,e}$ and

\begin{equation*}
\hat{\psi}(\hat{x},\hat{y})=\hat{w}(\hat{x},-1)g(\hat{y})
\end{equation*}
where $g$ is the polynomial of degree $p$ introduced in Lemma \ref{existencia-del-poli-g} then, $\hat{\psi} \in \Q_{p}(Q)$ and satisfies
\begin{align*}
\hat{\psi} |_{\hat{e}} &= \hat{w}, \\
\hat{\psi} | _{\partial Q \setminus \hat{e}} &= 0, \\
\|\hat{\psi}\|_{H^{0,\b}(Q)} &\leq C \G(\b+1) (p+1)^{-(1+\b)}\|\hat{w}\|_{H^{0,\b}(\hat{e})}.
\end{align*}
Hence, the proof concludes by defining
\begin{equation*}
\psi=\hat{\psi}\circ F_{K,e}^{-1}
\end{equation*}
\end{proof}

For each $V\in\N$, in the following theorem we introduce a local operator $I^{\b}_V:H_0^{1,\beta}(\O) \cap C^0(\o_V)\rightarrow S^{p}(\T_V)$
and we present some local error estimates.

\begin{theorem}\label{interpolador-en-omegaV}
Let $-1<\b<-1/2$ and $p$ a polynomial degree. For each  $V\in\N$  and $u \in H_0^{1,\beta}(\O) \cap C^0(\o_V)$ there exist an operator $I^{\b}_V:H_0^{1,\beta}(\O) \cap C^0(\o_V)\rightarrow S^{p}(\T_V)$  and a positive constant $C$ independent of $p$, $\b$ and $u$ such that
\begin{equation*}
\begin{aligned}
\|u-I^{\b}_V u\|_{H^{0,\b}(K)}&\leq  \frac{C}{-1-2\b} (p+1)^{-(3/2+\b)} |u|_{H^{1,\b}(\T_V)} \quad \forall K \in \T_V ,\\
\|u-I^{\b}_V u\|_{H^{0,\b}(e)}&\leq \frac{C}{-1-2\b} (p+1)^{-1/2}
|u|_{H^{1,\b}(\T_V)} \qquad \forall e \in \E_V^{\circ}.
\end{aligned}
\end{equation*}
If $\gamma\in  (\E \setminus \E^{\circ})$ is such  $\gamma \subset \partial \o_V$ then $I^{\b}_V u | _{\gamma}= u | _{\gamma} =0$.
\end{theorem}
\begin{proof}

For each $K\in\T_V$ we consider $\Pi^{\b}_{p,K} u$ as in Theorem \ref{cotas-enK}. Now, we will define $\phi_{K}$
such that $\phi_{K} (V)=u(V)$ for all $V$  vertex of $K$.
Indeed, given $V_1,V_2,V_3$ and $V_4$ the vertexes of $K$, we define the polynomial $\phi_{K}$ of degree $p$  as following:
\begin{equation*}
\phi_{K}= \Pi^{\b}_{p,K} u + \sum_{l=1}^4 (u-\Pi^{\b}_{p,K} u)(V_l)\xi_{K,l}
\end{equation*}
with $\xi_{K,l}$ the function defined in Corollary \ref{existencia-de-chi-l}.


Therefore, using (\ref{cota-interpolador-enK}),
(\ref{cota-interpolador-en-vertice-enK}) and Corollary \ref{existencia-de-chi-l} iii) we find that
\begin{equation*} 
\begin{aligned}
\|u-\phi_{K}\|_{H^{0,\b}(K)}&\leq \|u-\Pi^{\b}_{p,K} u\|_{H^{0,\b}(K)}+\sum_{l=1}^4 |(u-\Pi^{\b}_{p,K} u)(V_l)| \|\xi_{K,l}\|_{H^{0,\b}(K)} \nonumber \\
&\leq C(p+1)^{-1}|u|_{H^{1,\b}(K)}+C \frac1{(-1-2\b)}(p+1)^{-(3/2+\b)}|u|_{H^{1,\b}(K)} \nonumber \\
&\leq  \frac{C}{(-1-2\b)}(p+1)^{-(3/2+\b)}|u|_{H^{1,\b}(K)},
\end{aligned}
\end{equation*}
similarly, for any $e$ edge of $K$, from (\ref{cota-interpolador-en-lado-enK}),
(\ref{cota-interpolador-en-vertice-enK}) and  Corollary \ref{existencia-de-chi-l} iv), we have that
\begin{equation*}
\|u-\phi_{K}\|_{H^{0,\b}(e)}\leq \frac{C}{(-1-2\b)\G(\b+1)}(p+1)^{-1/2}|u|_{H^{1,\b}(K)}.
\end{equation*}

Let $e\in \E^{\circ}$ such that $e\subset \T_V\cap (\o_V)^{\circ}$ and  let $K_1$ and $K_2$ be the elements of $\T_V$ that share the edge $e$,
we consider
\begin{equation*}
w=(\phi_{K_1}  - \phi_{K_2}) |_e \quad \in
\P_{p}(e).
\end{equation*}
Observe that
\begin{equation*}
\begin{aligned}
\|w\|_{H^{0,\b}(e)}&\leq \|u-\phi_{K_1}\|_{H^{0,\b}(e)}+
\|u-\phi_{K_2}\|_{H^{0,\b}(e)}\\
&\leq \frac{C}{(-1-2\b)\G(\b+1)} (p+1)^{-1/2}(|u|_{H^{1,\b}(K_1)}+|u|_{H^{1,\b}(K_2)}).
\end{aligned}
\end{equation*}
Hence, let $\psi\in \Q_{p}(K_1)$ an extension of $w$ to $K_1$ as in Corollary \ref{extension} then,
\begin{align*}
\psi |_e &= w ,\\
\psi | _{\partial K_1 \setminus e} &= 0, \\
\|\psi\|_{H^{0,\b}(K_1)} &\leq C \G(\b+1) (p+1)^{-(1+\b)}\|w\|_{H^{0,\b}(e)}.
\end{align*}

We define $\tilde{\phi}_{K_1} =\phi_{K_1} -\psi$ and we note that this function satisfies:
\begin{align*}
\tilde{\phi}_{K_1} |_e &=\phi_{K_2}  |_e ,\\
\tilde{\phi}_{K_1} |_{\partial K_1 \setminus e} &= \phi_{K_1}|_{\partial K_1 \setminus e},
\end{align*}
and therefore we have the following error estimates

\begin{equation*}
\begin{aligned}
\|u-\tilde{\phi}_{K_1} \|_{H^{0,\b}(K_1)} &\leq
\|u-\phi_{K_1}\|_{H^{0,\b}(K_1)} +\|\psi\|_{H^{0,\b}(K_1)}\\
&\leq
\frac{C}{(-1-2\b)\G(\b+1)}(p+1)^{-(3/2+\b)}|u|_{H^{1,\b}(K_1)} +C \G(\b+1) (p+1)^{-(1+\b)}\|w\|_{H^{0,\b}(e)}\\
&\leq
\frac{C}{(-1-2\b)\G(\b+1)}(p+1)^{-(3/2+\b)}|u|_{H^{1,\b}(K_1)} \\
&+C \G(\b+1)(p+1)^{-(1+\b)}\frac{C}{(-1-2\b)\G(\b+1)}(p+1)^{-1/2}(|u|_{H^{1,\b}(K_1)}+|u|_{H^{1,\b}(K_2)})\\
&\leq
\frac{C}{-1-2\b}(p+1)^{-(3/2+\b)}(|u|_{H^{1,\b}(K_1)}+|u|_{H^{1,\b}(K_2)}),
\end{aligned}
\end{equation*}
and
\begin{equation*}
\begin{aligned}
\|u-\tilde{\phi}_{K_1} \|_{H^{0,\b}(e)} &\leq
\|u-\phi_{K_1}\|_{H^{0,\b}(e)} +\|\psi\|_{H^{0,\b}(e)}\\
&=\|u-\phi_{K_1}\|_{H^{0,\b}(e)}+\|w\|_{H^{0,\b}(e)}\\
&\leq \frac{C}{(-1-2\b)\G(\b+1)}(p+1)^{-1/2}|u|_{H^{1,\b}(K_1)} \\
&+\frac{C}{(-1-2\b)\G(\b+1)}(p+1)^{-1/2}(|u|_{H^{1,\b}(K_1)}+|u|_{H^{1,\b}(K_2)})\\
&\leq \frac{C}{(-1-2\b)\G(\b+1)}(p+1)^{-1/2}(|u|_{H^{1,\b}(K_1)}+|u|_{H^{1,\b}(K_2)}).
\end{aligned}
\end{equation*}

Then, we can modify $\phi_{K_1}$ to $\tilde{\phi}_{K_1}$ and repeat this process in each $e\in \E^{\circ}$ such that $e\subset \T_V\cap (\o_V)^{\circ}$. \\

If $\gamma\in  (\E \setminus \E^{\circ})$ is such  $\gamma \subset \partial \o_V$,  we consider $K\in\T_V$ such that $\gamma \in \partial K$. Let $w=\phi_{K}|_{\gamma}$, we observe that  $w=u=0$ in $\partial \gamma$. Now, let $\psi$ be  as in Corollary \ref{extension} an extension of $w$ to $K$ and  $\tilde{\phi}_{K} =\phi_{K}-\psi$ then,
\begin{align*}
\tilde{\phi}_{K} |_{\gamma} &=0, \\
 \tilde{\phi}_{K}  |_{\partial K \setminus \gamma} &= \phi_{K} |_{\partial K \setminus \gamma},
 \end{align*}
and
\begin{equation*}
\begin{aligned}
\|u-\tilde{\phi}_{K} \|_{H^{0,\b}(K)} &\leq
\|u-\tilde{\phi}_{K}\|_{H^{0,\b}(K)} +\|\psi\|_{H{0,\b}(K)}\\
&\leq
\frac{C}{(-1-2\b)\G(\b+1)}(p+1)^{-(3/2+\b)}|u|_{H^{1,\b}(K)} + C\G(\b+1) (p+1)^{-(1+\b)}\|w\|_{H^{0,\b}(\gamma)}\\
&\leq
\frac{C}{(-1-2\b)\G(\b+1)}(p+1)^{-(3/2+\b)}|u|_{H^{1,\b}(K)} \\
&+C\G(\b+1) (p+1)^{-(1+\b)}\big(\|u-\phi_{K}\|_{H^{0,\b}(\gamma)}+\|u\|_{H^{0,\b}(\gamma)}\big),
\end{aligned}
\end{equation*}
since  $\|u\|_{H^{0,\b}(\gamma)}=0$ because $u|_{\gamma}=0$ we have that
\begin{equation*}
\begin{aligned}
\|u-\tilde{\phi}_{K} \|_{H^{0,\b}(K)}&\leq
\frac{C}{(-1-2\b)\G(\b+1)}(p+1)^{-(3/2+\b)}|u|_{H^{1,\b}(K)} \\
&+C\G(\b+1) (p+1)^{-(1+\b)}\|u-\phi_{K}\|_{H^{0,\b}(\gamma)}\\
&\leq
\frac{C}{(-1-2\b)\G(\b+1)}(p+1)^{-(3/2+\b)}|u|_{H^{1,\b}(K)} \\
&+\frac{C}{(-1-2\b)} (p+1)^{-(1+\b)}(p+1)^{-1/2}|u|_{H^{1,\b}(K)}\\
&\leq
\frac{C}{-1-2\b}(p+1)^{-(3/2+\b)}|u|_{H^{1,\b}(K)} .
\end{aligned}
\end{equation*}
We can modify $\phi_K$ to $\tilde{\phi}_{K}$ and repeat the process for all $\gamma$ such that $\gamma\in  (\E \setminus \E^{\circ})$ and $\gamma \subset \partial \o_V$.
Thus, the operator $I^{\b}_V u |_K=\phi_{K}$ satisfies all the requirements.
\end{proof}

Now, we are in conditions to introduce a global operator $Iu \in S_0^{\mathbf{p}}(\T)$ which satisfies the following error estimates,
\begin{theorem}\label{interpolador-en-Omega}
Let  $-1<\b<-1/2$ and $u\in H_0^{1,\beta}(\O) \cap C^0(\O)$ then, there exist $Iu \in S_0^{\mathbf{p}}(\T)$ and a positive constant $C$ such that
\begin{equation}\label{interpolador-en-Omega-en-K}
\|u-Iu\|_{H^{0,\b}(K)}\leq\frac{C}{-1-2\b} p_K^{-(3/2+\b)} |u|_{H^{1,\b}(\T|_{\o_K})} \quad \forall \ K \in \T,
\end{equation}
\begin{equation}\label{interpolador-en-Omega-en-lado}
\|u-Iu\|_{H^{0,\b}(e)}\leq \frac{C}{-1-2\b} p_e^{-1/2} |u|_{H^{1,\b}(\T|_{\o_e})} \quad \forall \ e \in \E.
\end{equation}
The constant $C$ is independent of $u$, $\b$ and $\mathbf{p}$.
\end{theorem}

\begin{proof} A fundamental property of the space $S_0^{\mathbf{p}}(\T)$ (see, for example, \cite{Melenk2005}) is that we can identify ``nodal shape functions'' that form a partition of unity, i.e., for each vertex $V\in\N$, we can find a function $\varphi_V\in S^1(\T)$ such that
\begin{equation*}
\varphi_V |_{\O\setminus \o_V}\equiv 0, \quad  \sum_{V\in\N} \varphi_V \equiv 1 \quad \mbox{ and }  \quad \sup_{x \in \O} |\varphi_V(x)| \leq 1.
\end{equation*}
We consider $I^{\b}_V u\in S^{p_V-1}(\T_V)$ as in Theorem \ref{interpolador-en-omegaV} with $p=p_V -1$ and we define
\begin{equation*}
Iu=\sum_{V\in\N} \varphi_V I^{\b}_V u.
\end{equation*}
It is clear that $Iu \in C_0^0(\O)$ and
\begin{equation*}
Iu |_K = \sum_{V\in K} \varphi_V I^{\b}_V u |_K,
\end{equation*}
since $p_V\leq p_K$, for all $K$ such that $V \in K$, then $Iu |_K \in \Q_{p_K}(K)$ and
\begin{equation*}
\begin{aligned}
\|u-Iu\|_{H^{0,\b}(K)}&= \|\sum_{V\in \N} \varphi_V u-\sum_{V\in\N} \varphi_V I^{\b}_V u\|_{H^{0,\b}(K)}\\
&= \|\sum_{V\in \N} \varphi_V (u-I^{\b}_V u)\|_{H^{0,\b}(K)}
\leq  \sum_{V\in K} \| u-I^{\b}_V u \|_{H^{0,\b}(K)}\\
 & \leq \frac{C}{-1-2\b}\sum_{V\in K}p_V^{-(3/2+\b)} |u|_{H^{1,\b}(\T_V)},
\end{aligned}
\end{equation*}
since polynomial degrees of neighboring elements are comparable then $p_V^{-(3/2+\b)} \leq C p_K^{-(3/2+\b)}$, therefore
\begin{equation*}
\|u-Iu\|_{H^{0,\b}(K)}\leq \frac{C}{-1-2\b} p_K^{-(3/2+\b)} |u|_{H^{1,\b}(\T|_{\o_K})},
\end{equation*}
and the first estimate holds.
Now, let  $e\in \E$ we have that
\begin{equation*}
\begin{aligned}
\|u-Iu\|_{H^{0,\b}(e)}&= \|\sum_{V\in \N} \varphi_V u-\sum_{V\in\N} \varphi_V I^{\b}_V u\|_{H^{0,\b}(e)}\leq  \sum_{V\in e} \| u-I^{\b}_V u \|_{H^{0,\b}(e)} \\
&\leq\frac{C}{-1-2\b} \sum_{V\in e}p_V^{-1/2} |u|_{H^{1,\b}(\T_V)},
\end{aligned}
\end{equation*}
since polynomial degrees of neighboring elements are comparable then $p_V^{-1/2} \leq C p_e^{-1/2}$, therefore
\begin{equation*}
\|u-Iu\|_{H^{0,\b}(e)} \leq  \frac{C}{-1-2\b} p_e^{-1/2}    |u|_{H^{1,\b}(\T|_{\o_e})},
\end{equation*}
and we conclude the proof.
\end{proof}

In what follows we denote by $C_{\b}$ a generic constant with depends on $\b$ but is independent of $p$.

We finish this Section recalling the following estimates that will be useful later on.

\begin{theorem}\label{trazas-con-pesos} Let $K$ be a parallelogram in $ \R^2$, $\g$ an edge of $K$ and $-1<\b<0$.
There exist a  unique lineal and continuous function $T:H^{1,\b}(K)\rightarrow H^{0,\b}( \g) $ such that if $u\in C^{\infty}(\bar{K})$ then $T(u)=u$ and
\begin{equation*}
\|T(u)\|_{H^{0,\b}( \g)} \leq \Cb \|u\|_{H^{1,\b}(K)},
\end{equation*}
\end{theorem}
\begin{proof}
We show the case $K=Q=(-1,1)^2$. Let $u\in C^{\infty}(\bar{Q})$, by (\ref{expancion-Jacobi}) we have that
\begin{equation*}
u(x,y)=\sum_{i\geq 0, j\geq 0} c_{i,j} J_j^{\b}(x) J_i^{\b}(y).
\end{equation*}
We will bound only  $\|u(x,-1)\|_{H^{0,\b}(I)}$ with  $I=(-1,1)$, for the rest of the edges we can proceed similarly.
\begin{equation*}
u(x,-1)= \sum_{i\geq 0,j\geq
0} c_{i,j}J_i^{\b}(x)J_j^{\b}(-1)= \sum_{i\geq 0}c_{i,0}J_0^{\b}(-1)J_i^{\b}(x)+\sum_{i\geq 0}\sum_{j\geq 1}c_{i,j}J_j^{\b}(-1)J_i^{\b}(x)=I+II.
\end{equation*}
then, using (\ref{jacobi-en-uno}) we have that $J_0^{\b}(-1)=1$ and by (\ref{integral-derivadas-por-peso})
\begin{equation*}
\|I\|^2_{H^{0,\b}(I)} = \sum_{i\geq 0} |c_{i,0}|^2 \g_i^{\b}\leq (\g_0^{\b})^{-1}\sum_{i\geq 0, j\geq 0} |c_{i,j}|^2  \g_j^{\b} \g_i^{\b} \leq (\g_0^{\b})^{-1} |u|^2_{H^{0,\b}(Q)}
\end{equation*}
and
\begin{equation*}
\begin{aligned}
\|II\|^2_{H^{0,\b}(I)} &= \sum_{i\geq 0} \Big(\sum_{j\geq 1}c_{i,j}J_j^{\b}(-1)\Big)^2 \g_i^{\b}\leq \sum_{i\geq 0} \Big(\sum_{j\geq 1}|c_{i,j}| |J_j^{\b}(-1)|\Big)^2 \g_i^{\b}\\
&\leq \sum_{i\geq 0} \Big(\sum_{j\geq 1}|c_{i,j}|^2 \g_j^{\b}j^2\Big) \Big(\sum_{j\geq 1}|J_j^{\b}(-1)|^2(\g_j^{\b} )^{-1}j^{-2}\Big) \g_i^{\b}
\end{aligned}
\end{equation*}
using Lemma  \ref{estimacion-gamma}, (\ref{cota-jacobi-en-vertices}) and (\ref{serie-por-integral}) there exist a positive constant $\Cb$ such that
\begin{equation*}
\begin{aligned}
\|II\|^2_{H^{0,\b}(I)}&\leq\Cb \sum_{i\geq 0} \Big(\sum_{j\geq 1}|c_{i,j}|^2 \g_j^{\b}j^2\Big) \Big(\sum_{j\geq 1}(j+1)^{2\b}j^{-1}\Big) \g_i^{\b}\\
&\leq \Cb\sum_{i\geq 0} \Big(\sum_{j\geq 1}|c_{i,j}|^2 \g_j^{\b}j^2\Big) \Big(\int_1^{\infty} x^{2\b-1}dx\Big) \g_i^{\b}\\
&\leq  \Cb\sum_{i\geq 0} \Big(\sum_{j\geq 1}|c_{i,j}|^2 \g_j^{\b}j^2\Big) \Big( x^{2\b}|_1^{\infty}\Big) \g_i^{\b}\\
& \leq  \Cb \sum_{i\geq 0, j\geq 1}|c_{i,j}|^2 \g_j^{\b}j^2\g_i^{\b} \leq  \Cb |u|^2_{H^{1,\b}(Q)}.
\end{aligned}
\end{equation*}
Joining $I$ and $II$ we get
\begin{equation*}
\|u(x,-1)\|_{H^{0,\b}(I)}\leq \Cb  \|u\|_{H^{1,\b}(Q)}.
\end{equation*}
since $C^{\infty}(\bar{Q})$ is dense in $H^{1,\b}(Q)$  an unique continuous extension there exists and the proof concludes.
\end{proof}

For $u\in H^{1,\b}(K)$  and $\g$ an edge of $K$ we denote $\|u\|_{H^{0,\b}(\g)}=\|T(u)\|_{H^{0,\b}(\g)}$.

\begin{lemma}\label{estimaciones-inversas}
Let $I=(-1,1)$ and $\a,\b>-1$. Then, there exist $C_{1,\b}$ and $C_{2,\a}$ such that for all polynomials $P \in \P_{p}(I) $
\begin{equation*}
\begin{aligned}
\int_I P (x) ^2 (1-x^2)^{\b}dx&\leq  C_{1,\b} p^2 \int_I P (x) ^2 (1-x^2)^{\b+1}dx,\\
\int_I P'(x) ^2 (1-x^2)^{\a+1}dx&\leq  C_{2,\a} p^2 \int_I P (x) ^2 (1-x^2)^{\a}dx.
\end{aligned}
\end{equation*}
If, in addition, $-1<\a\leq \a_0$ and $-1/2\leq \b \leq -1/4$ we can choose $C_1$ and $C_2=C_2(\a_0)$ independent of $\b$ and $\a$.
\end{lemma}

\begin{proof}
The first inequality can be found in, e.g. \cite{BernardiMaday1997,BernardiFietierOwens2001}. The proof of the second inequality is analogous to the proof of Theorem 3.95 of \cite{Schwab1998}. Following the steps in those demonstrations we can see the dependence of $\b$ and $\a$ in the constants.
\end{proof}

\begin{lemma}\label{extension-del-lado}
Let  $I=(-1,1)$,  $-1<\b<1$ and $P \in \P_p(I)$, there exists $\hat{v}(\hat{x},\hat{y})$  which is defined on $Q$ with the following properties: 
\begin{itemize}
\item[i)] $\hat{v}(\hat{x},-1) = P (\hat{x})(1-\hat{x}^2)^{\b}$, $\hat{v}|_{\partial Q \setminus \ell}=0$ where $\ell = I\times \{-1\}$;
\item[ii)] $\|\hat{v}\|_{H^{0,-\b}(Q)}\leq  C_{\b} (p+1)^{\b-1}\|P \|_{H^{0,\b}(I)}$;
\item[iii)] $\|\hat{v}\|_{H^{1,-\b}(Q)}\leq  C_{\b} (p+1)^{\b}\|P \|_{H^{0,\b}(I)}$.
\end{itemize}
If, in addition, $1/2\leq \b \leq 3/4$ we can choose $C$ independent of $\b$.
\end{lemma}

\begin{proof}
By Lemma \ref{existencia-del-poli-g} we know that there exists $g\in\P_p(I)$ such that $g(1)=0$, $g(-1)=1$ and $\|g\|_{H^{0,-\b}(I)}\leq  C_{\b} (p+1)^{-(1-\b)}$. We define $\hat{v}(\hat{x},\hat{y})=P(\hat{x}) (1-\hat{x}^2)^{\b}g(\hat{y})$, these extension obviously satisfies condition i). To prove condition ii) we observe that
\begin{equation*}
\begin{aligned}
\int_Q \hat{v}(\hat{x},\hat{y})^2 (1-\hat{x}^2)^{-\b} (1-\hat{y}^2)^{-\b}&=\int_I P (\hat{x})^2 (1-\hat{x}^2)^{\b}d\hat{x} \int_{I}g(\hat{y})^2(1-\hat{y}^2)^{-\b} d\hat{y}\\
&=\|P \|^2_{H^{0,\b}(I)}\|g\|^2_{H^{0,-\b}(I)} \leq  C_{\b} (p+1)^{-2(1-\b)}\|P \|^2_{H^{0,\b}(I)},
\end{aligned}
\end{equation*}
Hence,
\begin{equation*}
\|\hat{v}\|_{H^{0,-\b}(Q)}\leq  C_{\b} (p+1)^{-(1-\b)}\|P \|_{H^{0,\b}(I)}.
\end{equation*}
To prove condition iii) we calculate $\|\frac{\partial \hat{v}}{\partial \hat{x}}\|_{H^{0,-\b}(Q)}$ and $\|\frac{\partial \hat{v}}{\partial \hat{y}}\|_{H^{0,-\b}(Q)}$
\begin{equation*}
\frac{\partial \hat{v}}{\partial \hat{x}} (\hat{x},\hat{y}) = \Big( \frac{\partial P}{\partial \hat{x}}(\hat{x}) (1-\hat{x}^2)^{\b}+P(\hat{x}) \b (1-\hat{x}^2)^{\b-1}(-2\hat{x}) \Big) g(\hat{y}),
\end{equation*}
then
\begin{equation*}
\begin{aligned}
\int_Q \big(\frac{\partial \hat{v}}{\partial \hat{x}} (\hat{x},\hat{y})\big)^2 (1-\hat{x}^2)^{1-\b}(1-\hat{y}^2)^{-\b} &\leq C \big\{ \int_Q \big( \frac{\partial P}{\partial \hat{x}} (\hat{x})\big)^2 (1-\hat{x}^2)^{\b+1}(1-\hat{y}^2)^{-\b}g(\hat{y})\\
&+\int_Q P (\hat{x})^2 (1-\hat{x}^2)^{\b-1}(1-\hat{y}^2)^{-\b}g(\hat{y}) \big\},
\end{aligned}
\end{equation*}
by inverse estimates  of Lemma \ref{estimaciones-inversas} we have that
\begin{equation*}
\begin{aligned}
 \int_I \big( \frac{\partial P}{\partial \hat{x}} (\hat{x})\big)^2 (1-\hat{x}^2)^{\b+1} d\hat{x} &\leq  C_{\b} (p+1)^2 \int_I P (\hat{x})^2 (1-\hat{x}^2)^{\b}d\hat{x}\\
\int_I P (\hat{x})^2 (1-\hat{x}^2)^{\b-1}d\hat{x} &\leq  C_{\b} (p+1)^2 \int_I P (\hat{x})^2 (1-\hat{x}^2)^{\b} d\hat{x}.
\end{aligned}
\end{equation*}
Therefore
\begin{equation*}
\begin{aligned}
\int_Q \big(\frac{\partial \hat{v}}{\partial \hat{x}} (\hat{x},\hat{y})\big)^2 (1-\hat{x}^2)^{1-\b}(1-\hat{y}^2)^{-\b} &\leq  C_{\b} (p+1)^2\| P\|^2_{H^{0,\b}(I)}\| g\|^2_{H^{0,-\b}(I)} \\
&\leq  C_{\b} (p+1)^{2}(p+1)^{-2(1-\b)}\| P\|^2_{H^{0,\b}(I)}\\
&=  C_{\b}(p+1)^{2\b}\| P\|^2_{H^{0,\b}(I)}.
\end{aligned}
\end{equation*}
On the other hand,
\begin{equation*}
\frac{\partial \hat{v}}{\partial \hat{y}} (\hat{x},\hat{y}) = P (\hat{x}) (1-\hat{x}^2)^{\b}\frac{\partial g}{\partial \hat{y}}(\hat{y}),
\end{equation*}
then
\begin{equation*}
\int_Q \big(\frac{\partial \hat{v}}{\partial \hat{y}} (\hat{x},\hat{y})\big)^2 (1-\hat{y}^2)^{1-\b}(1-\hat{x}^2)^{-\b} = \int_I P (\hat{x})^2 (1-\hat{x}^2)^{\b}d\hat{x}\int_I \left(\frac{\partial g}{\partial \hat{y}}(\hat{y})\right)^2 (1-\hat{y}^2)^{1-\b}d\hat{y},
\end{equation*}
by inverse estimates of Lemma \ref{estimaciones-inversas}
\begin{equation*}
\int_I\frac{\partial g}{\partial \hat{y}}(\hat{y})^2 (1-\hat{y}^2)^{1-\b}d\hat{y}\leq C_{\b}(p+1)^2 \int_I g(\hat{y})^2 (1-\hat{y}^2)^{-\b}d\hat{y},
\end{equation*}
thus
\begin{equation*}
\begin{aligned}
\int_Q \big(\frac{\partial \hat{v}}{\partial \hat{y}} (\hat{x},\hat{y})\big)^2 (1-\hat{y}^2)^{1-\b}(1-\hat{x}^2)^{-\b} &\leq  C_{\b} (p+1)^2 \|P \|_{H^{0,\b}(I)}^2\| g\|_{H^{0,-\b}(I)}\\
&\leq   
  C_{\b} (p+1)^{2\b} \| P \|_{H^{0,\b}(I)}^2 .
\end{aligned}
\end{equation*}
from which we conclude the proof.
\end{proof}

\section{A posteriori error estimation}\label{A posteriori}

In this section we introduce an a posteriori error indicator of the residual type for the classical Poisson model problem and, by using the $p$-interpolation error estimates obtained in the previous sections, we show the equivalence between the indicator and the error in an appropriate Jacobi-weighted norm up to higher order terms.

\subsection{Problem Statement}

 Let $\O\subset \R^2$ an open  polygonal domain, $\G=\partial \O$ and  $f\in L^2(\O) $. We consider the classical Poisson problem: Find a function  $u$ such that:
\begin{equation}
\label{Problema-Modelo} \left\{
\begin{aligned}
-\triangle u &= f \quad \mbox{in}\,\, \O   \\
u &= 0  \quad   \mbox{on}\, \, \G  \\
\end{aligned}\right.
\end{equation}

The  Variational Problem associated to  (\ref{Problema-Modelo}) is: Find
 $u\in H_0^1(\O)$ such that:
\begin{equation}\label{PV}
a(u,v)=L(v) \ \ \ \forall \ v\in H_0^1(\O)
\end{equation}
where $a(u,v) = \int_{\O} \nabla u \cdot \nabla v$  and $L(v)= \int_{\O} f v $.

Let $\T$ be an admissible partition of
$\O$ in parallelogram elements, and $S_0^{\mathbf{p}}(\T)$ as in (\ref{espacio-de-elementos-finitos}), the Discrete Variational Problem is  defined as follows: Find  $u\in S_0^{\mathbf{p}}(\T)$ such that:
\begin{equation}\label{PVD}
a(u,v)=L(v) \ \ \ \forall \ v\in S_0^{\mathbf{p}}(\T).
\end{equation}

Following the ideas introduced in \cite{Guo2005}, let $Q=(-1,1)^2$ the reference domain, 
$\b >-1$, $\a=(\a_1,\a_2)$ with $\a_i\geq 0$ integer, we define the weight function $\tilde{W}_{\b,\a}$ in $Q$ as follows:
\begin{equation*}
\tilde{W}_{\b,\a}(x_1,x_2)=(1-x_1^2)^{\b-\a_1}(1-x_2^2)^{\b-\a_2}.
\end{equation*}
Note that $\tilde{W}_{\b,\a} =  W_{\b,-\a}$.

Then, the weighted Sobolev space $\tilde{H}^{k,\b}(Q)$  is defined as the closure of the $C^{\infty}(\bar{Q})$ function with the norm
\begin{equation*}
\|u\|^2_{\tilde{H}^{k,\b}(Q)}=\sum_{|\a| \leq k} \int_Q |\p^{\a} u|^2 \tilde{W}_{\b,
\a},
\end{equation*}
by $|u|^2_{\tilde{H}^{k,\b}(Q)}$  we denote the semi-norm
\begin{equation*}
|u|^2_{\tilde{H}^{k,\b}(Q)}=\sum_{|\a|= k} \int_Q |\p^{\a} u|^2 \tilde{W}_{\b,
\a}.
\end{equation*}

Let $K$ be a parallelogram and let $F:Q\rightarrow K$ be an affine transformation. Given  $u\in
C^{\infty}(\bar{K})$ we can define $\tilde{u}=u\circ F \ \in
C^{\infty}(\bar{Q})$ and (see, for example \cite{Guo2005})
\begin{equation*}
\|u\|_{\tilde{H}^{k,\b}(K)}=\|\tilde{u}\|_{\tilde{H}^{k,\b}(Q)}.
\end{equation*}

Then, for $\T$  an admissible partition of
$\O$ in parallelogram elements and $u\in C^{\infty}(\bar{\O})$,
we define
\begin{equation*}
\|u\|^2_{\tilde{H}^{k,\b}(\T)}=\sum_{K\in\T}\|u\|^2_{\tilde{H}^{k,\b}(K)}.
\end{equation*}
Hence, the Jacobi-weighted spaces $\tilde{H}^{k,\b}(\T)$ and $\tilde{H}_0^{k,\b}(\T)$ for $\T$ a partition of $\O$ are defined as follows:
\begin{equation*}
\begin{aligned}
\tilde{H}^{k,\b}(\T)&=\overline{C^{\infty}(\bar{\O})}^{\|\cdot\|_{\tilde{H}^{k,\beta}(\T)}}\\
\tilde{H}_0^{k,\b}(\T)&=\overline{C_0^{\infty}(\bar{\O})}^{\|\cdot\|_{\tilde{H}^{k,\beta}(\T)}}.
\end{aligned}
\end{equation*}


\begin{lemma}\label{imbedding-1}
Let $0<\b<1$ and $u\in H^{1+s}(\O)$ with $s\geq \frac{1-\b}{2}$ then, $u\in \tilde{H}^{1,\b}(\T)$.
\end{lemma}
\begin{proof} Since the functions  $C^{\infty} (\bar{\O})$ are dense in $H^{1+s}(\O)$, the result can be obtained from the inequality (1.8) of \cite{BelgacemBrenner2001} on $Q$, i.e.,
\begin{equation}\label{cotaBB2001}
\int_{Q} \frac{f^2 (\hat{x},\hat{y})}{\rho((\hat{x},\hat{y}),\p Q)^{2\lambda}}  \leq C \|f\|^2_{H^{\lambda}(Q)} \quad \forall f\in  H^{\lambda}(Q)  \quad 0<\lambda<1/2
\end{equation}
where $\rho((\hat{x},\hat{y}),\p Q)$ denotes the distance from $(\hat{x},\hat{y})$ to the
boundary of $Q$.

Indeed, there exist  $\varphi _n \in C^{\infty} (\bar{\O})$ such that $\varphi _n \longrightarrow u $ as ${n\rightarrow \infty}$ in $H^{1+s}(\O)$.\\
Let  $K \in \T$ and $F_{K}:Q\rightarrow K$ be an affine transformation, we note $\widehat{\varphi _n -u}=(\varphi _n -u)\circ F_{K}$.\\
Since $\b>0$
\begin{equation*}
\int_Q (\widehat{\varphi _n -u}) ^2 (1-\hat{x}^2)^{\b}(1-\hat{y´}^2)^{\b}\leq \int_Q(\widehat{\varphi _n -u} )^2 \longrightarrow 0 \mbox { as } {n\rightarrow \infty},
\end{equation*}
 then
 \begin{equation*}
 \|\varphi _n -u\|_{H^{0,\b}(K)} \longrightarrow 0 \mbox { as } {n\rightarrow \infty}.
 \end{equation*}
Since $0<\b<1$ and  $(1-\hat{x}^2)\geq \rho((\hat{x},\hat{y}),\p Q)$ we get
\begin{equation*}
\int_Q \big(\partial_{\hat{x}}(\widehat{\varphi _n -u}) \big)^2 (1-\hat{x}^2)^{\b-1}(1-\hat{y}^2)^{\b}\leq \int_Q\frac{\big(\partial_{\hat{x}}(\widehat{\varphi _n -u}) \big)^2 }{\rho((\hat{x},\hat{y}),\p Q)^{2\l}}
\end{equation*}
with $\l=\frac{1-\b}{2}$. Now, since $0<\l<1/2$ by (\ref{cotaBB2001}) we follow that
\begin{equation*}
\int_Q\frac{\big(\partial_{\hat{x}}(\widehat{\varphi _n -u}) \big)^2 }{\rho((\hat{x},\hat{y}),\p Q)^{2\l}} \leq \|\partial_{x}(\widehat{\varphi _n -u})\|^2_{H^{\l}(Q)}\leq |\widehat{\varphi _n -u}|^2_{H^{1+\l}(Q)}
\end{equation*}
then, if $\l\leq s$ we have
\begin{equation*}
|\widehat{\varphi _n -u}|^2_{H^{1+\l}(Q)}\longrightarrow 0 \mbox { as } {n\rightarrow \infty},
\end{equation*}
 and therefore
\begin{equation*}
\int_Q \big(\partial_{\hat{x}}(\widehat{\varphi _n -u}) \big)^2 (1-\hat{x}^2)^{\b-1}(1-\hat{y}^2)^{\b} \longrightarrow 0 \mbox { as } {n\rightarrow \infty}.
\end{equation*}
We can proceed analogously for the other derivative and conclude that  $\varphi _n \longrightarrow u $ as $n\rightarrow \infty$ in $\tilde{H}^{1,\b}(K)$. Thus, the result follows since we can do this for any $K\in \T$.
\end{proof}

Let $u$ be the solution of (\ref{PV}) and let $u_N$ be the solution of (\ref{PVD}). As in \cite{Guo2005}, we introduce the norm for the error $e=u-u_N$ denoted by $|||e|||$ and $|||e|||_{K}$ by
\begin{equation*}
|||e|||_{K}= \sup_{\|v\|_{H^{1,-\b}(K)}=1} |a(e,v)|\leq \|e\|_{\tilde{H}^{1,\b}(K)}
\end{equation*}
and
\begin{equation*}
|||e|||= \sup_{\|v\|_{H_0^{1,-\b}(\T)}=1} |a(e,v)|\leq \|e\|_{\tilde{H}^{1,\b}(\T)}.
\end{equation*}

Let $0<\b<1$ and $u$ the solution of (\ref{PV}), if $u\in H^{1+s}(\O) $ with $s\geq \frac{1-\b}{2}$ then $e\in \tilde{H}^{1,\b}(\T)$  and,  $|||e|||_{K}$ and $|||e|||$ are well defined.

By (\ref{PV}) and (\ref{PVD}) we can infer that
\begin{equation}\label{ecuacion-error-en-S}
a(e,v_N)=0 \quad \forall v_N \in S_0^{\mathbf{p}}(\T).
\end{equation}
For each $l\in \E ^{\circ}$ we  choose a normal vector $n_l$ and denote $K_{in}$ and $K_{out}$ the elements that share the edge $l$. We define
$$ \Bl \frac{\partial u_N}{\partial n}\Br_{\ell} = \nabla (u_N |_{K_{out}})\cdot n_{\ell} - \nabla (u_N |_{K_{in}})\cdot n_{\ell},$$
which corresponds to the jump of the normal derivative of $u_N$ across
the edge $\ell$.

In what follows we assume that $1/2<\b<1$  and $u$, the solution of (\ref{PV}), is such that $u\in H^{1+s}(\O)\cap H_0^1(\O) $ with $s\geq \frac{1-\b}{2}$. Then, for $v\in H_0^{1,-\b}(\T)$ such that $\|v\|_{H_0^{1,-\b}(\T)}=1$
   we have that  $e\in\tilde{H}^{1,\b}(\T)$ and
\begin{equation}\label{ecuacion-error-en-H}
a(e,v)=\sum_{K\in\T} \Big( \int_K (f+\Delta u_N)v +\frac{1}{2} \sum_{\ell\subset \partial K \cap \E^{\circ}} \int_{\ell} \Bl \frac{\partial u_N}{\partial n} \Br_{\ell} v \Big).
\end{equation}
For $\e >0$ and $v_{\e}\in C_0^{\infty}(\O)$ such that $\|v-v_{\e}\|_{H_0^{1,-\b}(\T)}\leq \e$, by Theorem \ref{trazas-con-pesos} $\|v-v_{\e}\|_{H^{0,-\b}(\ell)}\leq \Cb \e$  for all $\ell \in \E$. 
Since $-1<-\b<-1/2$, if we take $I v_{\e}$ be as in Theorem \ref{interpolador-en-Omega}, by using (\ref{ecuacion-error-en-S}) and (\ref{ecuacion-error-en-H}) we get
\[
\begin{aligned}
a(e,v)&=a(e,v-Iv_{\e})+a(e,Iv_{\e})\\
&= \sum_{K\in\T} \Big( \int_K (f+\Delta u_N)(v-Iv_{\e}) +\frac{1}{2} \sum_{\ell\subset \partial K \cap \E^{\circ}} \int_{\ell} \Bl \frac{\partial u_N}{\partial n} \Br_\ell (v-Iv_{\e}) \Big)\\
&= \sum_{K\in\T} \Big( \int_K (f+\Delta u_N)\big((v-v_{\e})+(v_{\e}-Iv_{\e})\big)\\
&+\frac{1}{2} \sum_{\ell\subset \partial K \cap \E^{\circ}} \int_{\ell} \Bl \frac{\partial u_N}{\partial n} \Br_\ell \big((v-v_{\e})+(v_{\e}-Iv_{\e})\big)\Big)\\
&\leq \sum_{K\in\T} \Big( \|f+\Delta u_N\|_{H^{0,\b}(K)}\big(\|v-v_{\e}\|_{H^{0,-\b}(K)}+\|v_{\e}-Iv_{\e}\|_{H^{0,-\b}(K)}\big) \\
&+\frac{1}{2} \sum_{\ell\subset \partial K \cap \E^{\circ}} \Big\| \Bl \frac{\partial u_N}{\partial n} \Br_\ell\Big\|_{H^{0,\b}(\ell)}\big( \|v-v_{\e}\|_{H^{0,-\b}(\ell)}+\|v_{\e}-Iv_{\e}\|_{H^{0,-\b}(\ell)}\big) \Big).
\end{aligned}
\]
Hence, by (\ref{interpolador-en-Omega-en-K}), (\ref{interpolador-en-Omega-en-lado}),  we obtain
\begin{equation*}
\begin{aligned}
a(e,v)&\leq  \sum_{K\in\T} \Big( \|f+\Delta u_N\|_{H^{0,\b}(K)}\big(\e+\frac{C}{-1+2\b}p_K^{-(3/2-\b)}\|v_{\e}\|_{H^{1,-\b}(\T|_{\o_K})}\big) \\
&+\frac{1}{2} \sum_{\ell\subset \partial K \cap \E^{\circ}} \Big\| \Bl \frac{\partial u_N}{\partial n} \Br_\ell\Big\|_{H^{0,\b}(\ell)}\big(\Cb \e +\frac{C}{-1+2\b}p_\ell^{-1/2}\|v_{\e}\|_{H^{1,-\b}(\T|_{\o_\ell})}\big) \Big)\\
&\leq  \sum_{K\in\T} \Big( \|f+\Delta u_N\|_{H^{0,\b}(K)}\big(\e+\frac{C}{-1+2\b}p_K^{-(3/2-\b)}(\|v-v_{\e}\|_{H^{1,-\b}(\T|_{\o_K})}+\|v\|_{H^{1,-\b}(\T|_{\o_K})})\big)\\
&+\frac{1}{2} \sum_{\ell\subset \partial K \cap \E^{\circ}} \Big\| \Bl \frac{\partial u_N}{\partial n} \Br_\ell\Big\|_{H^{0,\b}(\ell)} \big(\Cb \e+\frac{C}{-1+2\b}p_\ell^{-1/2}(\|v-v_{\e}\|_{H^{1,-\b}(\T|_{\o_\ell})}+\|v\|_{H^{1,-\b}(\T|_{\o_\ell})})\big) \Big)\\
&\leq  \sum_{K\in\T} \Big( \|f+\Delta u_N\|_{H^{0,\b}(K)}\big(\e+\frac{C}{-1+2\b}p_K^{-(3/2-\b)}\e +\frac{C}{-1+2\b}p_K^{-(3/2-\b)}\|v\|_{H^{1,-\b}(\T|_{\o_K})}\big)\\
&+\frac{1}{2} \sum_{\ell\subset \partial K \cap \E^{\circ}} \Big\| \Bl \frac{\partial u_N}{\partial n} \Br_\ell\Big\|_{H^{0,\b}(\ell)} \big(\Cb \e+\Cb p_\ell^{-1/2}\e +\frac{C}{-1+2\b}p_\ell^{-1/2}\|v\|_{H^{1,-\b}(\T|_{\o_\ell})}\big) \Big).
\end{aligned}
\end{equation*}
 Since this holds for all $\e>0$  we conclude that
\begin{equation*}
a(e,v)\leq \frac{C}{-1+2\b} \sum_{K\in\T} \Big( \|f+\Delta u_N\|_{H^{0,\b}(K)}p_K^{-(3/2-\b)}\|v\|_{H^{1,-\b}(\T|_{\o_K})}+\frac{1}{2} \sum_{\ell\subset \partial K \cap \E^{\circ}} \Big\| \Bl \frac{\partial u_N}{\partial n} \Br_\ell\Big\|_{H^{0,\b}(\ell)}p_\ell^{-1/2}\|v\|_{H^{1,-\b}(\T|_{\o_\ell})} \Big).
\end{equation*}
Therefore,
\begin{equation*}
|||e|||\leq \frac{C}{-1+2\b} \sum_{K\in\T} \Big( \|f+\Delta u_N\|_{H^{0,\b}(K)}p_K^{-(3/2-\b)}+\frac{1}{2} \sum_{\ell\subset \partial K \cap \E^{\circ}} \Big\| \Bl \frac{\partial u_N}{\partial n} \Br_\ell \Big\|_{H^{0,\b}(\ell)}p_\ell^{-1/2}\Big).
\end{equation*}
Let  $f_{p_K}=\Pi^{\b}_{p_K,K} f$ as in Corollary \ref{cotas-enK}, then
\begin{equation*}
\begin{aligned}
|||e|||&\leq \frac{C}{-1+2\b} \sum_{K\in\T} \Big( \|f_{p_K}+\Delta u_N\|_{H^{0,\b}(K)}p_K^{-(3/2-\b)}+\|f-f_{p_K}\|_{H^{0,\b}(K)}p_K^{-(3/2-\b)}\\
&+\frac{1}{2} \sum_{\ell\subset \partial K \cap \E^{\circ}} \Big\| \Bl \frac{\partial u_N}{\partial n} \Br_\ell\Big\|_{H^{0,\b}(\ell)}p_l^{-1/2}\Big),
\end{aligned}
\end{equation*}
Let $0<\delta <1/2$, we take $\b=1/2+\delta$. Then
\begin{equation*}
\begin{aligned}
|||e|||&\leq \frac{C}{\delta} \sum_{K\in\T} \Big( \|f_{p_K}+\Delta u_N\|_{H^{0,\b}(K)}p_K^{-(1-\delta)}+\|f-f_{p_K}\|_{H^{0,\b}(K)}p_K^{-(1-\delta)}\\
&+\frac{1}{2} \sum_{\ell\subset \partial K \cap \E^{\circ}} \Big\| \Bl \frac{\partial u_N}{\partial n} \Br_\ell\Big\|_{H^{0,\b}(\ell)}p_\ell^{-1/2}\Big),
\end{aligned}
\end{equation*}
 $C$ independent of $\delta$.

 For each element $K\in \T$ the local error indicator is defined  as:
\begin{equation}\label{indicador-en-K}
\eta_K^2=\eta_{B_K}^2+ \eta_{E_K}^2,
\end{equation}
where
\begin{equation}\label{definiciones-indicadores-en-K}
\eta_{B_K}^2= \|f_{p_K}+\Delta u_N\|^2_{H^{0,\b}(K)}p_K^{-2} \quad\mbox{and}\quad
 \eta_{E_K}^2= \frac{1}{4} \sum_{\ell\subset \partial K \cap \E^{\circ}} \eta_\ell ^2,
\end{equation}
with
\begin{equation}\label{indicador-en-l}
\eta_\ell ^2=\| R_\ell\|^2_{H^{0,\b}(l)}p_\ell^{-1} , \quad  R_\ell = \Bl \frac{\partial u_N}{\partial n} \Br_\ell.
\end{equation}

Therefore, the global estimator is given by
\begin{equation}\label{indicador-global}
\eta ^2 = \sum _{K\in \T} \eta_K ^2.
\end{equation}

Thus, we have the following theorem
which proves the reliability of the error estimator up to higher order terms.
\begin{theorem}\label{reliability}
Let $\b=1/2+\delta$ with $0<\delta < \frac12$, $u$ the solution of (\ref{PV}), $u_N$ the solution of (\ref{PVD}) and $e=u-u_N$. Let $\eta$ be as in (\ref{indicador-global}). Assume that  $u\in H^{1+s}(\O)\cap H^1_0(\O)$ with $s\geq \frac{1-\b}{2}$. Then, there exists a positive constant $C$ such that
\begin{equation*}
\begin{aligned}
|||e|||&\leq \frac{C}{\delta} \max\{ p_{max}^{\delta},1\} \Big\{\eta+\Big(\sum_{K\in\T} p_K^{-2}\|f-f_{p_K}\|^2_{H^{0,\b}(K)}\Big)^{1/2}\Big\},
\end{aligned}
\end{equation*}
where $p_{max}=\max\{p_K | K\in\T\}$. The constant $C$ is independent of $\mathbf{p}$ and $\delta$.
\end{theorem}

In order to guarantee that the error indicator is efficient to guide an
adaptive refinement scheme, our next goal is to prove that $\eta_K$ is bounded by
the weighted norm of the error up to higher order terms.

The following lemma provides an upper estimate for the first term in the
definition of (\ref{indicador-en-K}).

\begin{theorem}\label{parte-vol-por-error-en-K}
Let $\b=1/2+\delta$ with $0<\delta < \frac14$, $u$ the solution of (\ref{PV}), $u_N$ the solution of (\ref{PVD}) and $e=u-u_N$. Let $\eta_{B_K}$  as in (\ref{definiciones-indicadores-en-K}). Assume that  $u\in H^{1+s}(\O)\cap H^1_0(\O)$  with $s\geq \frac{1-\b}{2}$. Then there exists a constant $C(K)$ such that
\begin{equation*}
\eta_{B_K}\leq C(K)\big(|||e|||_{K}+\frac1{p_K}\|f_{p_K}-f\|_{H^{0,\b}(K)} \big).
\end{equation*}
The constant $C(K)$ is not dependent on $p_K$ and $\delta$.
\end{theorem}

\begin{proof}
\begin{equation*}
\begin{aligned}
\|f_{p_K}+\Delta u_N\|^2_{H^{0,\b}(K)}& =\|\widehat{f_{p_K}+\Delta u_N}\|^2_{H^{0,\b}(Q)}\\
&=\int_{Q}(\widehat{ f_{p_K}+\Delta u_N})^2(1-\hat{x}^2)^{\b}(1-\hat{y}^2)^{\b}\\
&= \int_{Q}(\widehat{ f_{p_K}+\Delta u_N})\hat{v_K}
\end{aligned}
\end{equation*}

where $\hat{v_K}= (\widehat{ f_{p_K}+\Delta u_N})(1-\hat{x}^2)^{\b}(1-\hat{y}^2)^{\b} $,
then we define
\begin{equation*}
v_K=\hat{v_K}\circ F_{K}^{-1} \mbox{ in } K \quad\mbox{ and }\quad v_K=0 \mbox{ in } \O\setminus K.
\end{equation*}

Therefore
\begin{equation*}
\begin{aligned}
\|f_{p_K}+\Delta u_N\|^2_{H^{0,\b}(K)}&= C(K) \int_{K}(f_{p_K}+\Delta u_N)v_K\\
& =C(K)\big( \int_{K}f{v_K}+\int_{K}\Delta u_N {v_K}+\int_{K}(f_{p_K}-f){v_K}\big).
\end{aligned}
\end{equation*}

Now, since $v_K \in H^1_0(\O)$ and $u$ solution of (\ref{PV}) then $\int_{K}f{v_K}=\int_{K}\nabla u \cdot \nabla {v_K}$ and
\begin{equation*}
\begin{aligned}
\|f_{p_K}+\Delta u_N\|^2_{H^{0,\b}(K)}&=C(K)\Big(\int_{K}\nabla u \cdot \nabla {v_K}- \int_{K}\nabla u_N \cdot \nabla {v_K}+\int_{K}(f_{p_K}-f) {v_K}\Big)\\
&=C(K)\Big(\int_{K}\nabla e \cdot \nabla {v_K}+\int_{Q}(f_{p_K}-f) {v_K}\Big)\\
&=C(K)\Big(a(e,{v_K})+\int_{K}(f_{p_K}-f) {v_K}\Big)\\
&= C(K)\Big(\|{v_K}\|_{H^{1,-\b}(K)}a\big(e,\frac{{v_K}}{\|{v_K}\|_{H^{1,-\b}(K)}}\big)+\int_{K}(f_{p_K}-f) {v_K}\Big)\\
&\leq C(K)\big(\|{v_K}\|_{H^{1,-\b}(K)}|||e|||_{K}+\|f_{p_K}-f\|_{H^{0,\b}(K)} \|{v_K}\|_{H^{0,-\b}(K)}\big).
\end{aligned}
\end{equation*}

Observe that
\begin{equation}\label{cota-vK-0}
\begin{aligned}
\|{v_K}\|^2_{H^{0,-\b}(K)}&=\|\hat{v_K}\|^2_{H^{0,-\b}(Q)}=\int_Q \hat{v_K}^2 (1-\hat{x}^2)^{-\b}(1-\hat{y}^2)^{- \b}\\
&= \int_Q (\widehat{f_{p_K}+\Delta u_N})^2 (1-\hat{x}^2)^{\b}(1-\hat{y}^2)^{\b}=\|f_{p_K}+\Delta u_N\|^2_{H^{0,\b}(K)}.
\end{aligned}
\end{equation}

On the other hand,    
\begin{equation*}
|{v_K}|^2_{H^{1,-\b}}(K)=|\hat{v_K}|^2_{H^{1,-\b}}(Q)=\int_Q \big(\frac{\partial \hat{v_K}}{\partial \hat{x}}\big)^2(1-\hat{x}^2)^{1-\b}(1-\hat{y}^2)^{-\b}+\int_Q \big(\frac{\partial \hat{v_K}}{\partial \hat{y}}\big)^2(1-\hat{x}^2)^{-\b}(1-\hat{y}^2)^{1-\b}.
\end{equation*}
We note that
\begin{equation}\label{derivada-x-de-v_K-sombrero}
\frac{\partial \hat{v_K}}{\partial \hat{x}}=\Big(\frac{\partial (\widehat{f_{p_K}+\Delta u_N})}{\partial \hat{x}}(1-\hat{x}^2)^{\b}+(\widehat{f_{p_K}+\Delta u_N})\b(1-\hat{x}^2)^{\b-1}(-2\hat{x})\Big)(1-\hat{y}^2)^{\b},
\end{equation}
and therefore
\begin{equation*}
\begin{aligned}
\int_Q \big(\frac{\partial \hat{v_K}}{\partial \hat{x}}\big)^2(1-\hat{x}^2)^{1-\b}(1-\hat{y}^2)^{-\b}
&\leq C\Big(\int_Q \big(\frac{\partial (\widehat{f_{p_K}+\Delta u_N})}{\partial \hat{x}}\big)^2(1-\hat{x}^2)^{\b+1}(1-\hat{y}^2)^{\b}\\
&+\int_Q(\widehat{f_{p_K}+\Delta u_N})^2(1-\hat{x}^2)^{\b-1}(1-\hat{y}^2)^{\b}\big)\\
&= C (I+II).
\end{aligned}
\end{equation*}
Now, by the estimates given in Lemma \ref{estimaciones-inversas}, it follows that
\begin{equation*}
I\leq C p_K^2  \int_Q (\widehat{f_{p_K}+\Delta u_N})^2(1-\hat{x}^2)^{\b}(1-\hat{y}^2)^{\b}= C p_K^2 \|f_{p_K}+\Delta u_N\|^2_{H^{0,\b}(K)}
\end{equation*}
and
\begin{equation*}
II\leq C p_K^2  \int_Q (\widehat{f_{p_K}+\Delta u_N})^2(1-\hat{x}^2)^{\b}(1-\hat{y}^2)^{\b}= C p_K^2 \|f_{p_K}+\Delta u_N\|^2_{H^{0,\b}(K)}.
\end{equation*}
Therefore,
\begin{equation*}
\int_Q \big(\frac{\partial \hat{v_K}}{\partial \hat{x}}\big)^2(1-\hat{x}^2)^{1-\b}(1-\hat{y}^2)^{-\b}\leq C p_K^2 \|f_{p_K}+\Delta u_N\|^2_{H^{0,\b}(K)}.
\end{equation*}
Analogously
\begin{equation*}
\int_Q \big(\frac{\partial \hat{v_K}}{\partial \hat{y}}\big)^2(1-\hat{y}^2)^{1-\b}(1-\hat{x}^2)^{-\b}\leq C p_K^2 \|f_{p_K}+\Delta u_N\|^2_{H^{0,\b}(K)}.
\end{equation*}
Then,
\begin{equation}\label{cota-vK-1}
\|{v_K}\|_{H^{1,-\b}(K)}\leq C p_K \|f_{p_K}+\Delta u_N\|_{H^{0,\b}(K)}.
\end{equation}
Hence, we can conclude that
\begin{equation*}
\begin{aligned}
\|f_{p_K}+\Delta u_N\|_{H^{0,\b}(K)}&\leq C(K)\big(  p_K  |||e|||_{K}+\|f_{p_K}-f\|_{H^{0,\b}(K)} \big),
\end{aligned}
\end{equation*}
and the result follows at once.
\end{proof}

Next, we prove an upper estimate for $\eta_{\ell}$.

\begin{theorem}\label{parte-del-salto-por-error-en-l}
Let  $\b=1/2+\delta$ with $0<\delta < \frac14$, $u$ the solution of (\ref{PV}), $u_N$ the solution of (\ref{PVD}) and $e=u-u_N$. Let $\ell\in \E ^{\circ}$, $K_1$ and $K_2$ $\in T$ such that $\ell=K_1\cap K_2$ and  $\eta_\ell$ as in (\ref{indicador-en-l}). Assume that  $u\in H^{1+s} (\O)\cap H^1_0(\O)$ with $s\geq \frac{1-\b}{2}$. Then, there exists a constant $C(\ell)$ such that
\begin{equation*}
\begin{aligned}
\eta_\ell &\leq C(\ell)\Big(|||e|||_{K_1}p_\ell^{\delta} +\|f_{p_{K_1}}-f\|_{H^{0,\b}(K_1)}p_{K_1}^{-1+\delta}+|||e|||_{K_2}p_\ell^{\delta}\\
& +\|f_{p_{K_2}}-f\|_{H^{0,\b}(K_2)}p_{K_2}^{-1+\delta}+
|||e|||_{\T|_{\o_\ell}}p_\ell^{\delta}\Big).
\end{aligned}
\end{equation*}
The constant $C(\ell)$ is independent of $\mathbf{p}$ and $\delta$.
\end{theorem}

\begin{proof}

Let  $\o_\ell= K_1\cup K_2$ and $\hat{v}$ as in  Lemma \ref{extension-del-lado} with $P (\hat{x}) = \hat{{R_\ell}} (\hat{x},-1)$, $p=p_\ell$. Let $F_{K_1}:Q\rightarrow K_1$ and $F_{K_2}:Q\rightarrow K_2$  affine transformations such that $F_{K_i}(I\times\{-1\})=\ell$.  We define $v_\ell$ such that $v_\ell |_{K_i}=\hat{v}\circ F_{K_i}^{-1}$, then 
\begin{itemize}
\item[i)] $\widehat{(v_\ell |_\ell)}(\hat{x},-1) = \hat{{R_\ell}} (\hat{x})(1-\hat{x}^2)^{\b}, \quad  v_\ell|_{\partial \o_\ell \setminus \ell}=0$;
\item[ii)] $\|v_\ell\|_{H^{0,-\b}(\T|_{\o_\ell})}\leq C (p_\ell+1)^{-(1-\b)}\|R_\ell \|_{H^{0,\b}(\ell)}$;
\item[iii)] $\|v_\ell\|_{H^{1,-\b}(\T|_{\o_\ell})}\leq C (p_\ell+1)^{\b}\|R_\ell \|_{H^{0,\b}(\ell)}$.
\end{itemize}

Therefore
\begin{equation*}\begin{aligned}
\|R_\ell\|^2_{H^{0,\b}(\ell)}&=\int_I \hat{R_\ell}(\hat{x})^2(1-\hat{x}^2)^{\b}d\hat{x}= \int_I \hat{R_\ell}  \hat{v_\ell} =C(\ell)\int_{\ell} R_\ell v_\ell \\
&=C(\ell)\Big(-\int_{\o_\ell} \Delta u_N v_\ell-\int_{\o_\ell} \nabla u_N \cdot \nabla v_\ell\Big)\\
&=C(\ell)\Big(-\int_{K_1} \Delta u_N v_\ell-\int_{K_1} \nabla u_N \cdot \nabla v_\ell -\int_{K_2} \Delta u_N v_\ell-\int_{K_2} \nabla u_N \cdot \nabla v_\ell\Big)\\
&=C(\ell)\Big(-\int_{K_1} (f_{p_{K_1}}+\Delta u_N) v_\ell-\int_{K_1} \nabla u_N \cdot \nabla v_\ell +\int_{K_1} (f_{p_{K_1}}-f)v_\ell\\
&-\int_{K_2} (f_{p_{K_2}}+\Delta u_N) v_\ell-\int_{K_2} \nabla u_N \cdot \nabla v_\ell +\int_{K_2} (f_{p_{K_2}}-f)v_\ell+\int_{\o_\ell} f v_\ell\Big).
\end{aligned}
\end{equation*}
It is easy to see that  $v_\ell \in H_0^1(\o_\ell)$ then,
since   $u$ is a solution of (\ref{PV}) we have that $\int_{\o_\ell} f v_\ell=\int_{\o_\ell}\nabla u \cdot \nabla v_\ell$ and
\begin{equation}\label{cota-Rl}
\begin{aligned}
\|R_\ell\|^2_{H^{0,\b}(\ell)}&=C(\ell)\Big(-\int_{K_1} (f_{p_{K_1}}+\Delta u_N) v_\ell-\int_{K_1} \nabla u_N \cdot \nabla v_\ell +\int_{K_1} (f_{p_{K_1}}-f)v\\
&-\int_{K_2} (f_{p_{K_2}}+\Delta u_N) v_\ell-\int_{K_2} \nabla u_N \cdot \nabla v_\ell +\int_{K_2} (f_{p_{K_2}}-f)v_\ell+\int_{\o_\ell}\nabla u \cdot \nabla v_\ell\Big)\\
&=C(\ell)\Big(-\int_{K_1} (f_{p_{K_1}}+\Delta u_N) v_\ell +\int_{K_1} (f_{p_{K_1}}-f)v_\ell-\int_{K_2} (f_{p_{K_2}}+\Delta u_N) v_\ell\\
&+\int_{K_2} (f_{p_{K_2}}-f)v_\ell+\int_{\o_\ell} \nabla e \cdot \nabla v_\ell\Big),
\end{aligned}
\end{equation}
consequently
\begin{equation*}
\begin{aligned}
\|R_\ell\|^2_{H^{0,\b}(\ell)}
&\leq C(\ell)\Big(\|f_{p_{K_1}}+\Delta u_N\|_{H^{0,\b}(K_1)} \|v_\ell\|_{H^{0,-\b}(K_1)} +\|f_{p_{K_1}}-f\|_{H^{0,\b}(K_1)}\|v_\ell\|_{H^{0,-\b}(K_1)}\\
&+\|f_{p_{K_2}}+\Delta u_N\|_{H^{0,\b}(K_2)} \| v_\ell\|_{H^{0,-\b}(K_2)}+\|f_{p_{K_2}}-f\|_{H^{0,\b}(K_2)}\|v_\ell\|_{H^{0,-\b}(K_2)}\\
&+\|v_\ell\|_{H^{1,-\b}(\T|_{\o_\ell})}a( e, \frac{v}{\|v_\ell\|_{H^{1,-\b}(\T|_{\o_\ell})}})\Big).
\end{aligned}
\end{equation*}

From ii) and iii) we get

\begin{equation*}
\begin{aligned}
\|R_\ell\|_{H^{0,\b}(\ell}&\leq C(\ell)\Big(\|f_{p_{K_1}}+\Delta u_N\|_{H^{0,\b}(K_1)}(p_\ell+1)^{-1/2+\delta} +\|f_{p_{K_1}}-f\|_{H^{0,\b}(K_1)}(p_\ell+1)^{-1/2+\delta}\\
&+\|f_{p_{K_2}}+\Delta u_N\|_{H^{0,\b}(K_2)} (p_\ell+1)^{-1/2+\delta}+\|f_{p_{K_2}}-f\|_{H^{0,\b}(K_2)}(p_\ell+1)^{-1/2+\delta}\\
&+(p_\ell+1)^{1/2+\delta}|||e|||_{\T|_{\o_\ell}}\Big),
\end{aligned}
\end{equation*}
then
\begin{equation*}
\begin{aligned}
p_\ell^{-1/2}\|R_\ell\|_{H^{0,\b}(\ell)}&\leq C(\ell)\Big(\|f_{p_{K_1}}+\Delta u_N\|_{H^{0,\b}(K_1)}p_\ell^{-1+\delta} +\|f_{p_{K_1}}-f\|_{H^{0,\b}(K_1)}p_\ell^{-1+\delta}\\
&+\|f_{p_{K_2}}+\Delta u_N\|_{H^{0,\b}(K_2)} p_\ell^{-1+\delta}+\|f_{p_{K_2}}-f\|_{H^{0,\b}(K_2)}p_\ell^{-1+\delta}+|||e|||_{\T|_{\o_\ell}}p_\ell^{\delta}\Big).
\end{aligned}
\end{equation*}
Since the polynomial degrees of neighboring elements are comparable, from Theorem \ref{parte-vol-por-error-en-K} we can  deduce that
\begin{equation*}
\begin{aligned}
p_\ell^{-1/2}\|R_\ell\|_{H^{0,\b}(\ell)}&\leq C(\ell)\Big(|||e|||_{K_1}p_\ell^{\delta} +\|f_{p_{K_1}}-f\|_{H^{0,\b}(K_1)}p_{K_1}^{-1+\delta}+|||e|||_{K_2}p_\ell^{\delta}\\
& +\|f_{p_{K_2}}-f\|_{H^{0,\b}(K_2)}p_{K_2}^{-1+\delta}+|||e|||_{\T|_{\o_\ell}}p_\ell^{\delta}\Big),
\end{aligned}
\end{equation*}
and we conclude the proof.
\end{proof}

Now we may conclude the efficiency of the error indicator.
\begin{theorem}\label{efficiency}
 Let $\b=1/2+\delta$ with $0<\delta <1/4$, $u$  the solution of (\ref{PV}), $u_N$ the solution of (\ref{PVD}), $e=u-u_N$ and $\eta$ as in (\ref{indicador-global}).  Assume that  $u\in H^{1+s} (\O)\cap H^1_0(\O)$ with $s\geq \frac{1-\b}{2}$. Then there exists a constant $C$ independent of $\mathbf{p}$ and $\delta$ such that
\begin{equation*}
\eta\leq C \max\{ p_{max}^{\delta},1\} \Big\{|||e|||+\big( \sum_{K\in\T}p_K^{-2}\|f_{p_{K}}-f\|^2_{H^{0,\b}(K)}\big)^{1/2}\Big\}.
\end{equation*}
\end{theorem}

\begin{proof}
Let $K\in\T$, and $v_K$ as in the proof of Theorem \ref{parte-vol-por-error-en-K}, in this proof we show that
\begin{equation*}
\|f_{p_K}+\Delta u_N\|^2_{H^{0,\b}(K)}= C(K)a(e,v_K)+C(K)\int_{K}(f_{p_K}-f) {v_K},
\end{equation*}
then,
\begin{equation*}
\eta_{B_K}^2= a(e,(p_K+1)^{-2}C(K){v_K})+(p_K+1)^{-2}C(K)\int_{K}(f_{p_K}-f) {v_K},
\end{equation*}

let $v_{\T}=\sum_{K\in \T} (p_K+1)^{-2}C(K){v_K}$, from (\ref{cota-vK-0}) and (\ref{cota-vK-1}) we find that
\begin{equation*}
\begin{aligned}
\sum_{K\in\T} \eta_{B_K}^2&= a(e,v_{\T})+\sum_{K\in \T}(p_K+1)^{-2}C(K)\int_{K}(f_{p_K}-f) {v_K}\\
&=\|v_{\T}\|_{H^{1,-\b}(\T)}a(e,\frac{v_{\T}}{\|v_{\T}\|_{H^{1,-\b}(\T)}}) +\sum_{K\in \T}(p_K+1)^{-2}C(K)\int_{K}(f_{p_K}-f) {v_K}\\
&\leq |||e||| \sum_{K\in \T}(p_K+1)^{-2}C(K)\|{v_K}\|_{H^{1,-\b}(K)}\\
&+\sum_{K\in \T}(p_K+1)^{-2}C(K)\|f_{p_K}-f\|_{H^{0,\b}(K)} \|{v_K}\|_{H^{0,-\b}(K)}\\
&\leq C \big(|||e|||\sum_{K\in \T}(p_K+1)^{-1}\|f_{p_K}+\Delta u_N\|_{H^{0,\b}(K)}\\
&+\sum_{K\in \T}(p_K+1)^{-2}\|f_{p_K}-f\|_{H^{0,\b}(K)} \|f_{p_K}+\Delta u_N\|_{H^{0,\b}(K)}\big)\\
&\leq C \Big(|||e|||+\big(\sum_{K\in \T}(p_K+1)^{-2}\|f_{p_K}-f\|^2_{H^{0,\b}(K)}\big)^{1/2} \Big) \big(\sum_{K\in \T} (p_K+1)^{-2}\|f_{p_K}+\Delta u_N\|^2_{H^{0,\b}(K)}\big)^{1/2}\Big),
\end{aligned}
\end{equation*}
hence,
\begin{equation}\label{parte-vol-global-por-error}
\big(\sum_{K\in\T} \eta_{B_K}^2\big)^{1/2}\leq C \Big(|||e|||+\big(\sum_{K\in \T} (p_K+1)^{-2}\|f_{p_K}+\Delta u_N\|^2_{H^{0,\b}(K)}\big)^{1/2}\Big).
\end{equation}
Let $\ell\in\E ^{\circ}$, and $v_\ell$ as in the proof of Theorem \ref{parte-del-salto-por-error-en-l}, from (\ref{cota-Rl}) we have that
\begin{equation*}
\begin{aligned}
\|R_\ell\|^2_{H^{0,\b}(\ell)}&= a(e,C(\ell)v_\ell)+ C(\ell)\Big(-\int_{K_1} (f_{p_{K_1}}+\Delta u_N) v_\ell +\int_{K_1} (f_{p_{K_1}}-f)v_\ell\\
&-\int_{K_2} (f_{p_{K_2}}+\Delta u_N) v_\ell+\int_{K_2} (f_{p_{K_2}}-f)v_\ell\Big).
\end{aligned}
\end{equation*}
Then,
\begin{equation*}
\begin{aligned}
\sum_{\ell\in\E}\eta_\ell ^2 
&= \|\sum_{\ell\in\E}C(\ell)p_\ell^{-1}v_\ell\|_{H^{1,-\b}(\T)}
a\Big(e,\frac{\sum_{\ell\in\E}C(\ell)p_\ell^{-1}v_\ell}{\|\sum_{\ell\in\E}C(\ell)p_\ell^{-1}v_\ell\|_{H^{1,-\b}(\T)}}\Big)
\\
&+ \sum_{\ell\in\E}C(\ell)\Big(-\int_{K_1} (f_{p_{K_1}}+\Delta u_N) v_\ell p_\ell^{-1} +\int_{K_1} (f_{p_{K_1}}-f)v_\ell p_\ell^{-1}\\
&-\int_{K_2} (f_{p_{K_2}}+\Delta u_N) v_\ell p_\ell^{-1}+\int_{K_2} (f_{p_{K_2}}-f)v_\ell p_\ell^{-1}\Big)\\
&\leq |||e|||  \sum_{\ell\in\E}C(\ell)p_\ell^{-1}\|v_\ell\|_{H^{1,-\b}(\T|_{\o_\ell})}+ \sum_{\ell\in\E}C(\ell)\Big(\|f_{p_{K_1}}+\Delta u_N\|_{H^{0,\b}(K_1)}\\
&+\|f_{p_{K_1}}-f\|_{H^{0,\b}(K_1)}+\|f_{p_{K_2}}+\Delta u_N\|_{H^{0,\b}(K_2)} + \|f_{p_{K_2}}-f\|_{H^{0,\b}(K_2)}\Big)\|v_\ell\|_{H^{0,-\b}(\T|_{\o_\ell})} p_\ell^{-1}.
\end{aligned}
\end{equation*}
From the estimates for $v_\ell$,  given in ii) and iii) in the proof of Theorem \ref{parte-del-salto-por-error-en-l}, it follows that
\begin{equation*}
\begin{aligned}
\sum_{\ell\in\E}\eta_\ell ^2 &\leq C |||e||| \sum_{\ell\in\E}p_\ell^{-1/2}p_\ell^{\delta}\|R_\ell\|_{H^{0,\b}(\ell)}+ \sum_{\ell\in\E}\Big(\|f_{p_{K_1}}+\Delta u_N\|_{H^{0,\b}(K_1)}\\
&+\|f_{p_{K_1}}-f\|_{H^{0,\b}(K_1)}+\|f_{p_{K_2}}+\Delta u_N\|_{H^{0,\b}(K_2)} + \|f_{p_{K_2}}-f\|_{H^{0,\b}(K_2)}\Big)\|R_\ell\|_{H^{0,\b}(\ell)}p_\ell^{-1/2} p_\ell^{-1+\delta}\\
&\leq C p_{max}^{\delta}|||e||| \big(\sum_{\ell\in\E}\eta_\ell ^2\big)^{1/2} + C \sum_{\ell\in\E} p_\ell^{2(-1+\delta)}\big(\|f_{p_{K_1}}+\Delta u_N\|^2_{H^{0,\b}(K_1)}+\|f_{p_{K_1}}-f\|^2_{H^{0,\b}(K_1)}\\
  &+\|f_{p_{K_2}}+\Delta u_N\|^2_{H^{0,\b}(K_2)} + \|f_{p_{K_2}}-f\|^2_{H^{0,\b}(K_2)}\big)^{1/2}\big(\sum_{\ell\in\E} \eta_\ell^2\big)^{1/2}.
\end{aligned}
\end{equation*}

Since the polynomial degrees of neighboring elements are comparable we have that
\begin{equation*}
\big(\sum_{\ell\in\E}\eta_\ell ^2\big)^{1/2} \leq C \Big\{ p_{max}^{\delta}|||e|||+\Big( \sum_{K\in\T}p_K^{2\delta}\eta_K^2+p_K^{2(-1+\delta)}\|f_{p_{K}}-f\|^2_{H^{0,\b}(K)}\Big)^{1/2}\Big\}.
\end{equation*}
Then,
\begin{equation*}
\big(\sum_{\ell\in\E}\eta_\ell ^2\big)^{1/2} \leq C p_{max}^{\delta} \Big\{|||e|||+\big( \sum_{K\in\T}\eta_K^2\big)^{1/2}+\big( \sum_{K\in\T}p_K^{-2}\|f_{p_{K}}-f\|^2_{H^{0,\b}(K)}\big)^{1/2}\Big\},
\end{equation*}
 and by (\ref{parte-vol-global-por-error})
\begin{equation*}
\big(\sum_{\ell\in\E}\eta_\ell ^2\big)^{1/2} \leq C p_{max}^{\delta} \Big\{|||e|||+\big( \sum_{K\in\T}p_K^{-2}\|f_{p_{K}}-f\|^2_{H^{0,\b}(K)}\big)^{1/2}\Big\}.
\end{equation*}
Hence, we are in conditions to compute $\eta$.
\begin{equation*}
\begin{aligned}
\eta^2=\sum_{K\in\T} (\eta_{B_K}^2+\eta_{E_K}^2)&\leq C \big(\sum_{K\in\T}\eta_{B_K}^2+\sum_{\ell\in\E}\eta_\ell^2 \big)
&\leq C \max\{ p_{max}^{2\delta},1\}  \Big\{|||e|||+\big( \sum_{K\in\T}p_K^{-2}\|f_{p_{K}}-f\|^2_{H^{0,\b}(K)}\big)^{1/2}\Big\}^2,
\end{aligned}
\end{equation*}
as claimed.
\end{proof}

\bigskip

We end the paper by emphasizing that, as far as we know, the quasi-optimal estimates reached in Theorems \ref{reliability} and \ref{efficiency} are  the best results that can be obtained for error estimators of the residual type for the two dimensional case.

\section*{Acknowledgements}

The work of M.G. Armentano and V. Moreno was supported by ANPCyT under grant PICT 2010-01675 and by Universidad de Buenos Aires under grant 20020100100143.


\begin{thebibliography}{APRS12}

\bibitem[AP02]{AinsworthPinchedez2002}
{\scshape M.~Ainsworth {\normalfont \smfandname} K.~Pinchadez} -- {\og
  hp-approximation theory for bdfm and rt finite elements on
  quadrilaterals\fg}, \emph{SIAM J. Numer. Anal.} \textbf{40} (2002), no.~6,
  p.~2047--2068.

\bibitem[APRS11]{APRS2011}
{\scshape M.~G. Armentano, C.~Padra, R.~Rodr{\'{\i}}guez {\normalfont
  \smfandname} M.~Scheble} -- {\og An {$hp$} finite element adaptive scheme to
  solve the {L}aplace model for fluid-solid vibrations\fg}, \emph{Comput.
  Methods Appl. Mech. Engrg.} \textbf{200} (2011), no.~1-4, p.~178--188.

\bibitem[APRS12]{APRS2012}
\bysame , {\og An {$hp$} adaptive strategy to compute the vibration modes of a
  fluid-solid coupled system\fg}, \emph{CMES Comput. Model. Eng. Sci.}
  \textbf{84} (2012), no.~4, p.~359--381.

\bibitem[APS14]{APS2014}
{\scshape M.~G. Armentano, C.~Padra {\normalfont \smfandname} M.~Scheble} --
  {\og An hp finite element adaptive scheme to solve the poisson problem on
  curved domains\fg}, \emph{to appear in Computational and Applied Mathematics}
  (2014).

\bibitem[AS97]{AinsworthSenior1997}
{\scshape M.~Ainsworth {\normalfont \smfandname} B.~Senior} -- {\og Aspects of
  an adaptive {$hp$}-finite element method: adaptive strategy, conforming
  approximation and efficient solvers\fg}, \emph{Comput. Methods Appl. Mech.
  Engrg.} \textbf{150} (1997), no.~1-4, p.~65--87, Symposium on Advances in
  Computational Mechanics, Vol. 2 (Austin, TX, 1997).

\bibitem[AS98]{AinsworthSenior1998}
\bysame , {\og An adaptive refinement strategy for hp-finite element
  computations\fg}, \emph{Applied Numerical Mathematics} \textbf{26} (1998),
  p.~165--178.

\bibitem[BB01]{BelgacemBrenner2001}
{\scshape F.~B. Belgacem {\normalfont \smfandname} S.~C. Brenner} -- {\og Some
  nonstandard finite element estimates with applications to 3d poisson and
  signorini problems\fg}, \emph{Electronic Transactions on Numerical Analysis}
  \textbf{12} (2001), p.~134--148.

\bibitem[BD11]{BurgDorfler2011}
{\scshape M.~B{\"u}rg {\normalfont \smfandname} W.~D{\"o}rfler} -- {\og
  Convergence of an adaptive {$hp$} finite element strategy in higher
  space-dimensions\fg}, \emph{Appl. Numer. Math.} \textbf{61} (2011), no.~11,
  p.~1132--1146.

\bibitem[BFO01]{BernardiFietierOwens2001}
{\scshape C.~Bernardi, N.~Fi{\'e}tier {\normalfont \smfandname} R.~G. Owens} --
  {\og An error indicator for mortar element solutions to the {S}tokes
  problem\fg}, \emph{IMA J. Numer. Anal.} \textbf{21} (2001), no.~4,
  p.~857--886.

\bibitem[BG00]{BabuskaGuo2000}
{\scshape I.~Babu{\v{s}}ka {\normalfont \smfandname} B.~Guo} -- {\og Optimal
  estimates for lower and upper bounds of approximation errors in the
  {$p$}-version of the finite element method in two dimensions\fg},
  \emph{Numer. Math.} \textbf{85} (2000), no.~2, p.~219--255.

\bibitem[BG02a]{BabuskaGuo2001}
\bysame , {\og Direct and inverse approximation theorems for the p-version of
  the finite element method in the framework of weighted besov spaces. part i:
  Approximability of functions in the weighted besov spaces\fg}, \emph{SIAM J.
  Numer. Anal.} \textbf{39} (2001/2002), no.~5, p.~1512--1538.

\bibitem[BG02b]{BabuskaGuo2002}
\bysame , {\og Direct and inverse approximation theorems for the {$p$}-version
  of the finite element method in the framework of weighted {B}esov spaces.
  {II}. {O}ptimal rate of convergence of the {$p$}-version finite element
  solutions\fg}, \emph{Math. Models Methods Appl. Sci.} \textbf{12} (2002),
  no.~5, p.~689--719.

\bibitem[BG10]{BabuskaGuo2010}
\bysame , {\og Local jacobi operators and applications to the p-version of
  finite element method in two dimensions\fg}, \emph{SIAM J. Numer. Anal.}
  \textbf{48} (2010), no.~1, p.~147--163.

\bibitem[BM97]{BernardiMaday1997}
{\scshape C.~Bernardi {\normalfont \smfandname} Y.~Maday} -- {\og Spectral
  methods\fg}, Handbook of numerical analysis, {V}ol. {V}, Handb. Numer. Anal.,
  V, North-Holland, Amsterdam, 1997, p.~209--485.

\bibitem[BPS09]{BraessPillweinSchoberl2009}
{\scshape D.~Braess, V.~Pillwein {\normalfont \smfandname} J.~Sch{\"o}berl} --
  {\og Equilibrated residual error estimates are {$p$}-robust\fg},
  \emph{Comput. Methods Appl. Mech. Engrg.} \textbf{198} (2009), no.~13-14,
  p.~1189--1197.

\bibitem[DH07]{DorflerHeuveline2007}
{\scshape W.~D{\"o}rfler {\normalfont \smfandname} V.~Heuveline} -- {\og
  Convergence of an adaptive {$hp$} finite element strategy in one space
  dimension\fg}, \emph{Appl. Numer. Math.} \textbf{57} (2007), no.~10,
  p.~1108--1124.

\bibitem[DM13]{DorsekMelenk2013}
{\scshape P.~D{\"o}rsek {\normalfont \smfandname} J.~M. Melenk} -- {\og
  Symmetry-free, {$p$}-robust equilibrated error indication for the
  {$hp$}-version of the {FEM} in nearly incompressible linear elasticity\fg},
  \emph{Comput. Methods Appl. Math.} \textbf{13} (2013), no.~3, p.~291--304.

\bibitem[GB86]{GuiBabuska1986III}
{\scshape W.~Gui {\normalfont \smfandname} I.~Babu{\v{s}}ka} -- {\og The
  {$h,\;p$} and {$h$}-{$p$} versions of the finite element method in {$1$}
  dimension. {I}. {T}he error analysis of the {$p$}-version\fg}, \emph{Numer.
  Math.} \textbf{49} (1986), no.~6, p.~577--612.

\bibitem[GB13]{BabuskaGuo2013}
{\scshape B.~Guo {\normalfont \smfandname} I.~Babu{\v{s}}ka} -- {\og Direct and
  inverse approximation theorems for the {$p$}-version of the finite element
  method in the framework of weighted {B}esov spaces, part {III}: {I}nverse
  approximation theorems\fg}, \emph{J. Approx. Theory} \textbf{173} (2013),
  p.~122--157.

\bibitem[GS07]{GuoSun2007}
{\scshape B.~Guo {\normalfont \smfandname} W.~Sun} -- {\og The optimal
  convergence of the {$h\text{-}p$} version of the finite element method with
  quasi-uniform meshes\fg}, \emph{SIAM J. Numer. Anal.} \textbf{45} (2007),
  no.~2, p.~698--730.

\bibitem[Guo05]{Guo2005}
{\scshape B.~Guo} -- {\og Recent progress on a-posteriori error analysis for
  the {$p$} and {$h\text{-}p$} finite element methods\fg}, Recent advances in
  adaptive computation, Contemp. Math., vol. 383, Amer. Math. Soc., Providence,
  RI, 2005, p.~47--61.

\bibitem[Guo09]{Guo2009}
\bysame , {\og Aproximation theory for the p-version of the finite element
  method in three dimensions part ii: convergence of the p version of the
  finite element method\fg}, \emph{SIAM J. Numer. Anal.} \textbf{47} (2009),
  no.~4, p.~2578--2611.

\bibitem[GW04]{GuoWang2004}
{\scshape B.-y. Guo {\normalfont \smfandname} L.-l. Wang} -- {\og Jacobi
  approximations in non-uniformly {J}acobi-weighted {S}obolev spaces\fg},
  \emph{J. Approx. Theory} \textbf{128} (2004), no.~1, p.~1--41.

\bibitem[Lan99]{Lang1999}
{\scshape S.~Lang} -- \emph{Complex analysis}, fourth \smfedname, Graduate
  Texts in Mathematics, vol. 103, Springer-Verlag, New York, 1999.

\bibitem[Mel05]{Melenk2005}
{\scshape J.~M. Melenk} -- {\og {$hp$}-interpolation of nonsmooth functions and
  an application to {$hp$}-a posteriori error estimation\fg}, \emph{SIAM J.
  Numer. Anal.} \textbf{43} (2005), no.~1, p.~127--155.

\bibitem[MW01]{MelenkWohlmuth2001}
{\scshape J.~M. Melenk {\normalfont \smfandname} B.~I. Wohlmuth} -- {\og On
  residual-based a posteriori error estimation in {$hp$}-{FEM}\fg}, \emph{Adv.
  Comput. Math.} \textbf{15} (2001), no.~1-4, p.~311--331 (2002).

\bibitem[Sch98]{Schwab1998}
{\scshape C.~Schwab} -- \emph{{$p$}- and {$hp$}-finite element methods},
  Numerical Mathematics and Scientific Computation, The Clarendon Press Oxford
  University Press, New York, 1998, Theory and applications in solid and fluid
  mechanics.

\end{thebibliography}
%

\end{document}